\documentclass[11pt, a4paper, onecolumn]{article}

\usepackage{graphicx}%
\usepackage{multirow}%
\usepackage{amsmath,amssymb,amsfonts}%
\usepackage{amsthm}%
\usepackage{mathrsfs}%
\usepackage[title]{appendix}%
\usepackage{xcolor}%
\usepackage{textcomp}%
\usepackage{manyfoot}%
\usepackage{booktabs}%
\usepackage{algorithm}%
\usepackage{algorithmicx}%
\usepackage{algpseudocode}%
\usepackage{listings}%
\usepackage{subcaption}
\usepackage[shortlabels]{enumitem}
\usepackage{cleveref}
\usepackage{pifont}

\usepackage{natbib} 
\makeindex                  
\usepackage{float}

\usepackage[left=1.6cm, right=1.6cm, top=1.6cm, bottom=2cm]{geometry}

\title{\huge Geometric Milstein Scheme for Stochastic Differential Equations on SO(n) and SE(n)}

\author{%
    Xi Wang and
    Victor Solo\thanks{X. Wang and V. Solo are with the School of Electrical Engineering \& Telecommunications, University of New South Wales,  Australia (Email: wangxi14.ucas@gmail.com, v.solo@unsw.edu.au).} \thanks{V. Solo is the corresponding author.}
     \thanks{This work was supported by the Australian Research Council (ARC) under Grant DP220102159.}
}

\newcommand{\SOn}{\mathrm{SO}(n)}
\newcommand{\son}{\mathfrak{so}(n)}
\newcommand{\sonp}{\mathfrak{so}(n)^{\perp}}
\newcommand{\SEn}{\mathrm{SE}(n)}
\newcommand{\sen}{\mathfrak{se}(n)}
\newcommand{\senp}{\mathfrak{se}(n)^{\perp}}
\newcommand{\R}{\mathbb{R}}
\newcommand{\E}{\mathbb{E}}
\renewcommand{\SS}{\mathsf{S}}
\newcommand{\X}{\bf{X}}
\newcommand{\bfE}{\bf{E}}
\newcommand{\m}{m}
\newcommand{\D}{\tg {D}}

\newcommand{\LO}{\tg L^{O}}
\newcommand{\LE}{\tg L^{E}}
\newcommand{\DO}{\tg D^{O}}
\newcommand{\DE}{\tg D^{E}}
\newcommand{\SSO}{\nm S^{O}}
\newcommand{\SSE}{\nm S^{E}}

\newcommand{\bfR}{\mathbf{R}}
\newcommand{\Z}{\mathcal{Z}}
\newcommand{\B}{\mathcal{B}}
\newcommand{\hatB}{\widehat{\B}}
\newcommand{\hatZ}{\widehat{\Z}}
\newcommand{\hatR}{\widehat{\bfR}}
\newcommand{\hatE}{\widehat{\bf E}}
\newcommand{\hatX}{\widehat{\bf X}}
\newcommand{\hatD}{\widehat{\tg D}}
\newcommand{\hatv}{\widehat{v}}
\newcommand{\hatp}{\widehat{\bf p}}
\newcommand{\nmC}{\mathsf{C}}
\renewcommand{\bf}[1]{\mathbf{#1}}
\newcommand{\tg}[1]{\mathcal{#1}}
\newcommand{\nm}[1]{\mathsf{#1}}

\newcommand{\calO}{\mathcal{O}}

\newcommand{\ms}{{m^*}}
\newcommand{\mos}{{(m+1)^*}}

\newcommand{\BIms}{{\B_{I,\m}}}
\newcommand{\Bjms}{{\B_{j,\m}}}
\newcommand{\hBIms}{{\hatB_{I,\m}}}
\newcommand{\hBjms}{{\hatB_{j,\m}}}
\newcommand{\hBjpms}{{\hatB_{j',\m}}}
\newcommand{\Dms}{{\Delta^{*}_{m}}}
\newcommand{\hDjjms}{ {\hatD_{j,j}^{\m}} }
\newcommand{\hDjjpms}{ {\hatD_{j,j'}^{\m}} }
\newcommand{\hDjpjms}{ {\hatD_{j',j}^{\m}} }
\newcommand{\hDllms}{ {\hatD_{l,l}^{\m}} }
\newcommand{\hDllpms}{ {\hatD_{l,l'}^{\m}} }
\newcommand{\njjms}{ {\nabla_{j,j}^{\m}} }
\newcommand{\njjpms}{ {\nabla_{j,j'}^{\m}} }
\newcommand{\hnjjms}{ {\widehat{\nabla}_{j,j}^{\m}} }
\newcommand{\hnjjpms}{ {\widehat{\nabla}_{j,j'}^{\m}} }

\newcommand{\hBIRms}{{\hatB^{\bfR}_{I,\m}}}
\newcommand{\hBjRms}{{\hatB^{\bfR}_{j,\m}}}
\newcommand{\hBjpRms}{{\hatB^{\bfR}_{j',\m}}}
\newcommand{\hDjjRms}{ {\hatD_{j,j}^{\bfR,\m}} }
\newcommand{\hDjjpRms}{ {\hatD_{j,j'}^{\bfR,\m}} }

\newcommand{\hnjjRms}{ {\widehat{\nabla}_{j,j}^{\bfR,\m}} }
\newcommand{\hnjjpRms}{ {\widehat{\nabla}_{j,j'}^{\bfR,\m}} }

\newcommand{\hBIPms}{{\hatB^{\bf p}_{I,\m}}}
\newcommand{\hBjPms}{{\hatB^{\bf p}_{j,\m}}}
\newcommand{\hBjpPms}{{\hatB^{\bf p}_{j',\m}}}

\newcommand{\hDjjpPms}{ {\hatD_{j,j'}^{\bf p,\m}} }

\newcommand{\hnjjpPms}{ {\widehat{\nabla}_{j,j'}^{\bf p,\m}} }

\newcommand{\Ijjpms}{ {I_{j,j'}^{\ms}} }
\newcommand{\Ijpjms}{ {I_{j',j}^{\ms}} }
\newcommand{\tIjjpms}{ { \tilde I_{j,j'}^{\ms}} }
\newcommand{\DIjjpms}{\Delta{I_{j,j'}^{\ms}}}
\newcommand{\Illpms}{ {I_{l,l'}^{\ms}} }

\newcommand{\tIllpms}{ { \tilde I_{l,l'}^{\ms}} }
\newcommand{\DIllpms}{\Delta{I_{l,l'}^{\ms}}}

\newcommand{\sumub}{{\lfloor \frac{t}{\delta} \rfloor}}

\newcommand{\hZms}{ \hatZ_{\ms} }
\newcommand{\hZmos}{ \hatZ_{\mos} }
\newcommand{\hZmpos}{ \hatZ_{(m'+1)^*} } 
\newcommand{\half}{{\frac{1}{2}}}
\newcommand{\thalf}{{\tfrac{1}{2}}}
\newcommand{\hRms}{{\hatR_\m}}
\newcommand{\hRmps}{{\hatR_{(m')}}}
\newcommand{\Rms}{{\bfR_\m}}

\newcommand{\ermn}[1]{\varepsilon_m^{(#1)}}
\newcommand{\eprm}{\varepsilon_m'}
\newcommand{\eprmn}[1]{\varepsilon_m^{(#1)'}}
\newcommand{\epprm}{\tilde\varepsilon_m}
\newcommand{\epprmn}[1]{\tilde\varepsilon_m^{(#1)}}

\newcommand{\suptpt}{\sup_{0 \le t'\le t}}

\newcommand{\ajr}{a_{j,r}}
\newcommand{\ajpr}{a_{j',r}}
\newcommand{\bjr}{b_{j,r}}
\newcommand{\bjpr}{b_{j',r}}

\newcommand{\Xr}{X_r}
\newcommand{\Phjms}{\Phi_{j,\m}}

\DeclareMathOperator{\POtg}{\Pi^O_{T_{\bf R}}}
\DeclareMathOperator{\POnm}{\Pi^O_{N_{\bf R}}}
\DeclareMathOperator{\PEtg}{\Pi^E_{T_{\E}}}
\DeclareMathOperator{\PEnm}{\Pi^E_{N_{\E}}}

\DeclareMathOperator{\tr}{tr}

\newcommand{\Der}[4]{
  #1_{#2}\bigl(#3(#4)\bigr)
}

\newcommand{\DerDOjjp}[1]{\Der{\DO}{\B_j}{\B_{j'}}{#1}}

\newcommand{\DerDOjpj}[1]{\Der{\DO}{\B_{j'}}{\B_{j}}{#1}}
\newcommand{\DerNjjp}[1]{\Der{\nabla}{\B_j}{\B_{j'}}{#1}}

\newcommand{\EsupN}[1]{\E\Bigl[ \sup_{0 \le t\le T} \big\| \sum_{m=0}^{\sumub} #1 \big\|^2 \Bigr]}

\newtheorem{assumption}{Assumption}

\newtheorem{theorem}{Theorem}

\newtheorem{example}{Example}

\newtheorem{lemma}{Lemma}

\crefname{section}{Section}{Sections}
\Crefname{section}{Section}{Section}
\crefname{algorithm}{Algorithm}{Algorithms}
\Crefname{algorithm}{Algorithm}{Algorithms}
\crefname{theorem}{Theorem}{Theorems}
\Crefname{theorem}{Theorem}{Theorems}
\crefname{lemma}{Lemma}{Lemmas}
\Crefname{lemma}{Lemma}{Lemmas}
\crefname{definition}{Definition}{Definitions}
\Crefname{definition}{Definition}{Definitions}
\crefname{example}{Example}{Examples}
\Crefname{example}{Example}{Examples}
\crefname{remark}{Remark}{Remarks}
\Crefname{remark}{Remark}{Remarks}
\crefname{assumption}{Assumption}{Assumptions}
\Crefname{assumption}{Assumption}{Assumptions}
\crefname{equation}{}{}
\Crefname{equation}{Equation}{Equations}
\crefname{appendix}{Appendix}{Appendices}
\Crefname{appendix}{Appendix}{Appendices}
\date{}
\begin{document}

\maketitle

\begin{abstract}
    In the paper, we propose a higher-order geometry-preserving numerical method for stochastic differential equations (SDEs) evolving on the Lie groups SO(n) and SE(n). Most existing Lie group integrators rely on Magnus expansion of the exponential map, which makes the construction of higher-order stochastic schemes difficult. To overcome this limitation, we develop a tangent-space parameterization corrected Milstein method (TaSP-CM), extending the tangent space parameterization (TaSP) framework from Lie-group ODEs to the stochastic setting. Although TaSP is a well-established method for Lie ODEs, the extension to SDEs is non-trivial and requires new stochastic corrections that ensure both geometric consistency and higher-order accuracy. We prove that the proposed scheme achieves strong convergence of order 1 under both commutative and non-commutative noise. Numerical experiments illustrate the theoretical results and demonstrate the efficiency and robustness of the proposed method.
\end{abstract}

\section{INTRODUCTION}

    Over the past few decades, stochastic differential equations (SDEs) on Lie groups have found increasing application. Among these, the special orthogonal group $\mathrm{SO}(n)$ and the special Euclidean group $\mathrm{SE}(n)$ play a critical role due to their broad applications in satellite control~\cite{brockett1973lie,hsu2002stochastic}, robot pose estimation~\cite{kwon2007particle}, object tracking~\cite{soatto1996motion}, generative model~\cite{bortoli2022riemannian,park2022riemannian}, geoscience, and particle physics~\cite{david2022geometric,forest2006geometric}.

The broad range of applications involving SDEs on Lie groups have led to an increasing need for numerical methods that respect the underlying geometric constraints. In this paper, we consider the class of nonlinear autonomous It\^o SDEs evolving on Lie groups $G = \SOn,\SEn$, given by
\begin{align}\label{eq:I-SDE-X}
    & d\bf X(t) = \bf X \tg B_{I}(\bf X(t)) dt + \half\bf X \nm \Sigma_N(\X) dt+  \bf X\sum_{j=1}^d \B_{j}(\bf X(t)) dW_j(t),
\end{align}
where $dW_j(t)$ are independent standard scalar Wiener processes, the drift $\bf X\tg B_I$ and diffusion fields $\bf X\tg B_j$ lie in the tangent space $T_{\bf X}G$, and the curvature-related pinning-drift term $\bf X\Sigma_N(\X)$ lies in the normal space $N_{\bf X}G$~\cite{solo2024stratonovich}. Our goal is to develop geometry-preserving numerical schemes (GPNS) for~\cref{eq:I-SDE-X} with strong convergence guarantees. A GPNS with time step size $\delta$ produces a discrete-time approximation $\bf{\hatX}(t):[0,T]\to G$ such that, for every noise realization of $dW_j(t)$, the numerical trajectory remains on $G$ at all times. A GPNS is said to achieve strong convergence rate (SCR) $p$, if the squared expected mean-square error between $\bf{\hatX}(t)$ and the exact solution $\X(t)$ is $\calO(\delta^p)$.

When it is reduced to the flat Euclidean setting, i.e., when $G = \mathbb{R}^n$ and $\nm \Sigma(\X) \equiv 0$, there exists numerous literature on numerical schemes with SCR, including the Euler–Maruyama method, Milstein method, and Runge–Kutta methods~\cite{maruyama1955continuous,mil1975approximate,ripley2009stochastic}. However, these Euclidean methods are inherently designed to exploit the linear structure of Euclidean space. When applied directly to Lie groups, where a global vector space structure is absent, these methods fail to preserve the underlying geometry and tend to drift off the manifold even over short time scales.

Some geometry-preserving schemes are proposed for the deterministic counterparts of Lie group SDEs, namely ordinary differential equations (ODEs) on Lie groups~\cite{munthekass1999high,munthe1998runge,hairer2006geometric}. However, the SDE~\cref{eq:I-SDE-X} includes a pinning drift that lies in the normal space, while ODE schemes account only for tangent dynamics. As a result, ODE schemes cannot be directly extended to the stochastic setting, which makes the design of GPNS for SDEs on $\SOn$ and $\SEn$ a non-trivial task.

Several prior works have introduced GPNS for SDEs on Lie groups based on the Magnus expansion. For instance, \cite{misawa2001lie} proposed the Lie operator splitting (L-OP) scheme for autonomous nonlinear SDEs over Lie groups and established a weak convergence rate. \cite{malham2008stochastic} developed the Geometric Castell–Gaines (G-CG) scheme with truncated Magnus expansion, demonstrating a strong convergence rate. However, these convergence results rely on the exact solution of certain Euclidean ODEs, which is computationally intractable in practice. In addition, \cite{marjanovic2018numerical} proposed an computationally tractable Euler–Maruyama method on Lie groups based on the Magnus expansion, but without a convergence analysis.

\begin{table}[t]
    \centering
    \label{tab:gpns-compare}
    \vspace{0.5em}
    \begin{tabular*}{\textwidth}{@{\extracolsep\fill}lcccccc}
        \toprule
        \textbf{Method} 
        & \textbf{GP} 
        & \textbf{CT} 
        & \textbf{Diffusion Type} 
        & \textbf{Strong conv. rate} \\
        \midrule
        Eu--Mil~\cite{mil1975approximate} 
            & \ding{55} & \ding{52} & N & $\mathcal{O}(\delta)$ \\
        L--SP~\cite{misawa2001lie} 
            & \ding{52} & \ding{55} & N & Weak \\
        G--CG~\cite{malham2008stochastic} 
            & \ding{52} & \ding{55} & N & $\mathcal{O}(\delta^{p})$ \\
        \midrule
        G--EM~\cite{piggott2016geometric,cheng2022theory} 
            & \ding{52} & \ding{52} & L & $\mathcal{O}(\delta^{1/2})$ \\
        S--RKMK~\cite{muniz2022higher,muniz2023strong} 
            & \ding{52} & \ding{52} & C/L & $\mathcal{O}(\delta^{k})$ \\
        G--EM, TaSP--EM~\cite{solo2024stratonovich,wang2025tangent} 
            & \ding{52} & \ding{52} & N & $\mathcal{O}(\delta^{1/2-\epsilon})$ \\
        \midrule
        \textbf{TaSP--CM (ours)} 
            & \ding{52} & \ding{52} & N & $\boldsymbol{\mathcal{O}(\delta)}$ \\
        \bottomrule
    \end{tabular*}
        \caption{Comparison of existing numerical schemes and our method. 
    GP: geometry-preserving; CT: computationally tractable; 
    Diffusion type: C = constant, L = linear, N = nonlinear; 
    $p$: truncation order of the Magnus expansion; 
    $k$: Runge--Kutta order.}
\end{table}

There are few computationally tractable GPNS schemes with strong convergence guarantees on $\SOn$ or $\SEn$. \cite{piggott2016geometric} established a Geometric Euler–Maruyama (G-EM) scheme for linear SDEs, i.e., with state-independent diffusion $\B_j$, on $\SOn$ and $\SEn$, proving a SCR of $1/2$. This result was further extended in \cite{cheng2022theory} to linear SDEs on general Riemannian manifolds, also achieving $1/2$ SCR. Furthermore, \cite{solo2024stratonovich} analyzed the G-EM method for nonlinear autonomous SDEs on $\SOn$ and established a SCR of $1/2 - \epsilon$ for arbitrarily small $\epsilon > 0$. Another approach, the TaSP-EM method, based on the Tangent Space Parametrization (TaSP) strategy~\cite{wang2025tangent}, was also shown to attain a SCR of $1/2 - \epsilon$. However, all these methods are Euler–Maruyama variants and are limited to at most $1/2$ SCR.

To the best of our knowledge, the only existing computationally tractable GPNS with SCR exceeding $1/2$ on $\SOn$ and $\SEn$ is proposed in~\cite{muniz2022higher,muniz2023strong}. By applying $p$‑th order Runge–Kutta methods to the Magnus expansion truncated at order $2p - 2$, the authors in~\cite{muniz2022higher} proposed the stochastic Runge–Kutta–Munthe–Kaas (S-RKMK) scheme and established a SCR of order $p$ for linear SDE on Lie groups. Furthermore, \cite{muniz2023strong} extended the $p$ SCR of S-RKMK to the nonlinear autonomous SDEs on Lie groups. However, the S-RKMK scheme in~\cite{muniz2023strong} is based on the Stratonovich formulation of the SDE. As shown in~\cite{solo2024stratonovich}, Stratonovich-based numerical schemes converge to the correct solution only when all the diffusion terms $\B_j$ are constants. Therefore, constructing computationally efficient GPNS with SCR greater than $1/2$ for nonlinear SDEs on $\SOn$ and $\SEn$ remains an open problem. Notably, a fundamental method in the Euclidean setting, the Euclidean Milstein (Eu-Mil)~\cite{mil1975approximate} scheme which achieves SCR $1$ for nonlinear SDEs, has yet to be extended to the Lie group setting.

Motivated by this research gap, we propose the Tangent Space Parametrization–Corrected Milstein (TaSP-CM) scheme, which extends the Euclidean Milstein method to the Lie groups $\SOn$ and $\SEn$. Our approach builds on the Tangent Space Parametrization (TaSP) technique introduced in~\cite{rabier1995numerical},  and achieves SCR $1$ and geometry perseverance simultaneously. The main contributions of this work are as follows:
\begin{itemize}
    \item We first introduce the TaSP-CM scheme for SDEs on the special orthogonal group~$\SOn$. By working in the Lie algebra and employing a Lie bracket correction, the scheme attains SCR $1$ under both commutative and non-commutative noise, which improves the $1/2$ SCR of geometry-preserving Euler–Maruyama schemes.
    \item  Leveraging the semidirect-product structure of $\SEn = \SOn \ltimes \mathbb{R}^{n}$, we develop the TaSP framework on $\SEn$ and extend TaSP-CM over $\SEn$.  A convergence analysis also provides SCR $1$ convergence under the both commutative and non-commutative noise settings, as same as those on $\SOn$.
    \item We conduct numerical simulations on $\SOn$ and $\SEn$ and compare the performance with several baselines. The results not only illustates the theoretical analysis but also demonstrate the effectiveness and advantages of the proposed TaSP-CM method.
\end{itemize}
Table 1 summarizes the comparison between our method and prior works.

The remainder of the paper is organized as follows.  \Cref{sec:prelim} reviews the preliminaries for $\SOn$ and $\SEn$.  \Cref{sec:problem} formulates the It\^o SDEs on Lie groups and defines strong convergence for geometry-preserving schemes.  \Cref{sec:tasp-so} presents the TaSP–CM integrator on $\SOn$, and \Cref{sec:conv-so} proves that it attains SCR $1$ under both commutative and non-commutative noise.  \Cref{sec:tasp-se} extends the framework and analysis to $\mathrm{SE}(n)$.  Numerical results are given in \Cref{sec:experiments}.  Finally, \Cref{sec:conclusions} concludes the paper. Detailed proofs of our main results are set out in Appendices~\ref{app:tech}-\ref{app:conv-sen}.
\section{Preliminaries}\label{sec:prelim}
In this section we review the geometry of $\SOn$ and $\SEn$ with the associated calculus. For further details, please see~\cite{gallier2020differential,lee2018introduction}.
\subsection{Special Orthogonal Group SO(n)}
In this subsection, we introduce the geometric and calculus facts on $\SOn$.
\subsubsection{Geometry of SO(n)}
We define the special orthogonal group
\begin{align*}
   \SOn = \bigl\{\bf R\in\mathbb R^{n\times n} \bigl| \bf R^{\top}\mathbf R = I,  \det\bf R = 1\bigr\}, 
\end{align*}
as the set of all real $n\times n$ rotation matrices. The associated Lie algebra is the collection of skew‑symmetric matrices $$\son = \bigl\{\tg Z\in\mathbb R^{n\times n} \bigl| \tg Z^{\top}=-\tg Z\bigr\},$$
which stands for the infinitesimal rotation and forms a linear subspace of $\mathbb R^{n\times n}$. The orthogonal complement of $\son$ in $\mathbb R^{n\times n}$ is the collection of all symmetric matrices $$\sonp:=\{\nmC \in R^{n\times n} \mid \nmC^T = \nmC\}.$$

For a fixed rotation $\bf R\in\SOn$, the tangent and normal space are obtained by left‑multiplying the Lie algebra and the orthogonal complement, i.e., 
\begin{align*}
    T_{\bf R}\SOn = \{\bf R\tg Z \mid \tg Z\in\son\}, \qquad
    N_{\bf R}\SOn = \{\mathbf R \nm C \mid \nm C \in \sonp \}.
\end{align*}
 The orthogonal projection of a matrix $A \in \mathbb{R}^{n \times n}$ onto the tangent space $T_{\bf R}\SOn$ and normal space $N_{\bf R}\SOn$ is given by
\begin{align*}
\POtg(A) = \half\bigl(A-\bf R A^{\top}\bf R\bigr), \qquad
\POnm (A) = \half\bigl(A + \bf R A^{\top}\bf R\bigr).
\end{align*}

To help readers distinguish geometrical objects, we adopt the following typographic convention throughout the paper: boldface symbols (e.g., $\mathbf R$) represent points on the Lie groups; calligraphic symbols (e.g., $\mathcal Z$) denote elements in tangent spaces; and sans‑serif symbols (e.g., $\mathsf C$) indicate elements in normal spaces.

\subsubsection{Calculus on SO(n)}\label{subsec:co-der}

Equipped with the inner product 
$$g(\bfR \Z_1,\bfR \Z_2):=\operatorname{tr}\bigl(\Z_1^{\top}\Z_2\bigr), \Z_1,\Z_2\in\son,$$ 
$\SOn$ becomes a embedding Riemannian manifold in the Euclidean space $\mathbb R^{n\times n}$.
The embedding Riemannian structure allows us to define intrinsic calculus operators on $\SOn$~\cite{gallier2020differential}.

The Levi–Civita connection $\tg \LO$ acts as the intrinsic differentiation on tangent vector fields over $\SOn$ that is compatible with metric. Formally, let $\bfR \Z_{1}(\bfR)$ and $\bfR \Z_{2}(\bfR)$ be two tangent vector fields over $\SOn$. The Levi–Civita connection of $\bfR\Z_1$ along $\bfR\Z_2$ is defined as the tangent projection of the Euclidean directional derivative, 
\begin{align}\label{eq:LC-SO}
    \Der{\LO}{\bfR \tg Z_1}{\bfR \tg Z_2}{\bfR} :=  \POtg\bigl( \Der{\nabla}{\bfR\tg Z_1}{\bfR\tg Z_2}{\bfR} \bigr) =\bfR(\nabla_{\bf R\tg Z_1} \tg Z_2 + \half[\tg Z_1, \tg Z_2]),
\end{align}
where $\nabla$ is the Euclidean directional derivative, i.e.,
\begin{align*}
   \bf R\nabla_{ \bf R\tg Z_1} \tg Z_2 = & \bf R \sum_{r,s=1}^n \frac{\partial\tg Z_2}{\partial \bf R_{rs}} (\bf R\tg Z_2)_{rs} \in T_{\bf R}{\SOn},
\end{align*}
and the Lie bracket of $\son$ is defined as $[\Z_1,\Z_2] = \Z_1\Z_2 - \Z_2\Z_1$.

\Cref{eq:LC-SO} illustrates that, in addition to the standard directional derivative, a correction term $\half \bfR[\Z_{1},\Z_{2}]$ arises due to the non-commutative Lie algebra structure. In contrast, in Euclidean space, where vector fields commute, the differentiation reduces to the usual directional derivative. This contrast highlights a fundamental distinction between differentiation on the curved manifold $\SOn$ and in flat Euclidean space. With a slight abuse of language, we refer to
\begin{align*}
    \Der{\DO}{\tg Z_1}{\tg Z_2}{\bfR} := \nabla_{\bf R\tg Z_1} \tg Z_2 + \half[\tg Z_1, \tg Z_2]
\end{align*}
as the covariant derivative (CovD) of $\Z_{2}$ along the direction $\Z_{1}$. Moreover, we also focus on the normal residual of $\nabla_{\bf R\tg Z_1} (\bf R\tg Z_2)$, i.e.,
\begin{align*}
      \POnm \Big( \Der{\nabla}{\bfR\tg Z_1}{\bfR\tg Z_2}{\bfR}\Big)=\half \bf R(\Z_1\Z_2+\Z_2\Z_1),
\end{align*}
where we refer the following bilinear map $\SSO(\tg Z_1, \tg Z_2)(\bfR) =\half(\Z_1\Z_2+\Z_2\Z_1)$ as the second fundamental form (SFF)~\cite{lee2018introduction}.

\subsection{Special Euclidean Group SE(n)}
The special Euclidean group $\SEn$, the associated Lie algebra $\sen$, and the normal component $\senp$ are defined as
\begin{align*}
    \SEn = & \Bigl\{ \bfE =(\bfR,\bf p) =
      \begin{bmatrix}
         \bfR & \bf p \\
          0 & 1
      \end{bmatrix}
       \Bigm| \bfR\in\SOn,  \bf p\in\mathbb R^{n} \Bigr\}, \\
    \sen  = & 
   \Bigl\{\tg V \ = (\Z , v) =
      \begin{bmatrix}
         \Z & v \\
          0 & 0
      \end{bmatrix}
       \Bigm| 
      \Z \in \son,
      v \in\mathbb R^{n}
   \Bigr\}, \\
    \senp  = & 
   \Bigl\{\nm U \ = (\nm C , \nm b) =
      \begin{bmatrix}
        \nm C & 0 \\
        \nm b^{\top} & a
      \end{bmatrix}
       \Bigm| 
      \nm C \in \sonp,
      \nm b \in\mathbb R^{n},a\in\R
   \Bigr\}.
\end{align*}
Here $\bfR$ is the rotational component and $\bf p$ the translational component of a rigid-body pose $\bfE$, $\Z$ and $v$ give the corresponding infinitesimal rotation and translation. For a fixed pose $\bfE = (\bfR,\bf p)\in\SEn$, the tangent and normal spaces are $T_{\bfE}\SEn  =   \{ \bfE\tg V \mid \tg V\in\sen \}$ and $N_{\bfE}\SEn  =   \{ \bfE\nm U \mid \nm U\in\senp \}$

Writing a matrix
$
   A=\begin{bmatrix}A_{11}&\mathbf a\\ \mathbf b^{\top}&a_{22}\end{bmatrix},
$
the orthogonal projections to the tangent and normal spaces are
\begin{align*}
   \PEtg(A)  = 
   \begin{bmatrix}
     \POtg(A_{11})& \bf a\\
      0 & 0
   \end{bmatrix},
   \qquad 
   \PEnm(A)
    = 
   \begin{bmatrix}
     \POnm(A_{11}) &  0\\
      b^{\top} & a_{22}
   \end{bmatrix}.
\end{align*}

Similarly, equipping $\sen$ with the Riemannian metric,$ \langle\tg V_1,\tg V_2\rangle = \half\operatorname{tr}\bigl(\Z_1^{\top}\Z_2\bigr) + v_1^{\top} v_2$, for
\[
\bfE=
\begin{bmatrix}\bf R & \bf p \\ 
0 & 1\end{bmatrix}\in\SEn,
\qquad
\tg V_i(\bfE)=
\begin{bmatrix}
   \tg Z_i(\bfE) & v_i(\bfE)\\
    0 & 0
\end{bmatrix},
\quad i=1,2,
\]
the Levi–Civita connection on $\SEn$ is given by
\begin{align*}
   &\Der{\LE}{\bfE\tg V_1}{\bfE\tg V_2}{\bfE} =  \PEtg \bigl( \Der{\nabla}{\bfE\tg V_1}{\bfE\tg V_2}{\bfE} \bigr) =  \bfE \Der{\DE}{\tg V_1}{\tg V_2}{\bfE} 
\end{align*}
and the CovD on $\SEn$ is given by
\begin{align}\label{eq:LC-SE}
   & \Der{\DE}{\tg V_1}{\tg V_2}{\bfE} =
   \begin{pmatrix}
        \nabla_{\bfE\tg V_1}\tg Z_2 +\tfrac12[\tg Z_1,\tg Z_2] \\
        \nabla_{\bfE\tg V_1}v_2 + \tfrac12 \tg Z_1v_2
   \end{pmatrix}^{\top}. 
\end{align}
Finally, the SFF reads $\SSE(\tg V_1,\tg V_2)(\bfE):= (\SSO(\Z_1,\Z_2),0)$.

\subsection{Notation on Time Differences}

To better describe the piecewise solution of SDEs on $\SOn$, we introduce the following time difference notation. For an integer $ m $, a step size $ \delta $, and time $ t>0 $, we define the clipped time $ t_{\ms} = \min\{t, m\delta\} $ and set
\begin{align*}
 & \Dms t = t_{\mos} - t_{\ms}, \quad F_m = F(m\delta) \\ 
 & F_{\ms}(t) = F(t_\ms), \quad \Dms F(t) = F_{\mos} - F_{\ms},
\end{align*}
for any  (deterministic or stochastic) process $F(t)$. In addition, for Wiener process $W_j$ and $W_{j'}$, we denote the double It\^o integral within the clipped time as
\begin{align*}
   \Ijjpms(t) = \int_{t_\ms}^{t_\mos} \int_{t_\ms}^{s_1}dW_{j}(s_1) dW_{j'}(s_2).
\end{align*}
When $j = j'$, the double It\^o integral can be simplified as $\half\big((\Dms W_j)^2-\Dms t\big)$.

\section{Problem Statement}\label{sec:problem}
In this section, we formalize the It\^o SDEs on $\SOn$ and $\SEn$. We then define the notion of strong convergence for GPNS and identify key difficulties in designing GPNS schemes on $\SOn$ and $\SEn$.

\subsection{Stochastic Differential Equations on Lie Groups}
Let $G\subseteq\mathbb R^{n\times n}$ be a matrix Lie group with Lie algebra $\mathfrak g$.
An Itô SDE on $G$ can be written in the following It\^o form 
\begin{align}\label{eq:ito-conversion}
    & d\X(t) = \X\B_{s}(\X(t)) dt + \frac{1}{2} \sum_{j=1}^d  \Big(\Der{\nabla}{\X\B_j}{\X\B_j}{\X}\Big) dt + \X \sum_{j=1}^d \B_{j}(\X) dW_j(t),
\end{align}
where the coefficients $\B_s,\mathcal B_j\colon G\to\mathfrak g$ are smooth $\mathfrak g$-valued vector fields and $\{W_j(t)\}_{j=1}^{d}$ are independent standard Wiener processes. The middle term is the usual It\^o correction, expressed as a Euclidean directional derivative~\cite{le2016brownian}. Following the decomposition in~\cref{sec:prelim}, we split the It\^o correction into its tangent component (in $\mathfrak g$) and its normal component (in $\mathfrak g^{\perp}$) via the CovD $\mathcal D^{G}$ and the SFF $\mathcal S^{G}$, i.e.,
\begin{align*}
    \tg B_{I}(\X) := \B_0(\X) + \half \sum_j \Der{\D^G}{\B_j}{\B_j}{\X} \in \mathfrak g,  \quad 
    \nm \Sigma_N(\X):=  \sum_{j=1}^d  \SS^G(\B_j, \B_j)(\X)                      \in \mathfrak g^\perp.
\end{align*}
Equation~\eqref{eq:ito-conversion} then becomes
\begin{align}\label{eq:Ito-SDE-sec3}
    & d\bf X(t) = \bf X \tg B_{I}(\bf X(t)) dt + \half\bf X \nm \Sigma_N(\X) dt+  \bf X\sum_{j=1}^d \B_{j}(\bf X(t)) dW_j(t).
\end{align}

For $\mathbf X=\mathbf R\in\SOn$ with coefficients $\B_j(\mathbf R)\in\son$, we have $\nm \Sigma_N(\bf X)=\sum_{j=1}^d \B_j^2$. Hence, \cref{eq:Ito-SDE-sec3} becomes
\begin{align}\label{eq:I-SDE-son}
    & d\bf R(t) = \bf R \tg B_{I}(\bf R) dt + \half\mathbf R(t) \sum_{j=1}^{d} \B_{j}^{2}(\bfR)dt+  \bfR \sum_{j=1}^d \B_{j}(\bfR) dW_j(t).
\end{align}

Furthermore, for $\mathbf X=\mathbf E = (\mathbf R,\mathbf p) \in \SEn$ with each coefficient splitting as
$\mathcal B_{j}=(\B^{\bfR}_{j},\B^{\bf p}_{j})$, we have $\nm \Sigma_{N} (\bfE)= (\sum_{j=1}^d \B_j^2 , 0)$. Consequently, \cref{eq:Ito-SDE-sec3} becomes the coupled system
\begin{align}\label{eq:I-SDE-sen}
    \begin{cases}
        d\mathbf R(t) =\mathbf R(t)\B^{\mathbf R}_{I}(\bfE) dt +\half\mathbf R(t) \sum_{j=1}^{d} \bigl(\B^{\mathbf R}_{j}(\bfE)\bigr)^{2}dt +\mathbf R(t)\sum_{j=1}^{d} \B^{\mathbf R}_{j}(\bfE) dW_{j}(t), \\
        d\mathbf p(t) = \mathbf R(t)\B^{\bf p}_{I}(\bfE) dt + \mathbf R(t)\sum_{j=1}^{d}\B^{\bf p}_{j}(\bfE)dW_{j}(t).
    \end{cases}
\end{align}

The It\^o formulations in \cref{eq:ito-conversion,eq:I-SDE-son,eq:I-SDE-sen} highlight a fundamental distinction between Euclidean SDEs and Lie ODEs.  In addition to the tangential drift and diffusion terms, a SDE on $\SOn$ or $\SEn$ contains a normal component $\half\bf X \nm \Sigma_N(\X)$. The normal drifting term, known as pinning drift~\cite{solo2024stratonovich}, tends to pull the SDE away from $G$ and introduces a fundamental difficulty for GPNS on $\SOn$ and $\SEn$.

\subsection{Geometry‑Preserving Numerical Scheme}
In this paper, we aim to develop a geometry‑preserving numerical scheme (GPNS) that solve~\cref{eq:Ito-SDE-sec3} on $G=\SOn, \SEn$ with strong convergence rate $1$. In a GPNS, the interval $[0,T]$ is partitioned into $M$ equal sub‑intervals of length $\delta = T/M$. Starting from the initial condition $\widehat{\mathbf X}_0=\mathbf X(0)\in \mathrm{SO}(n)$, a GPNS evaluates the drift and diffusion terms at time $t=m\delta$ and constructs a local approximation on the interval $[m\delta,(m+1)\delta]$. To preserve the geometry, the local approximation is required to remain on $G$ for every realization of the diffustion noise.

A GPNS is said to achieve strong convergence rate (SCR) of $p$, if there exists a constant $C>0$, independent of $\delta$, such that the strong error
\begin{align*}
    \mathbb E\Bigl[\sup_{0 \le t \le T} \Vert \hatX(t) - \X(t) \Vert_{F}^{2} \Bigr]^{1/2} \le \calO(\delta^{ p}),
\end{align*}
where $\Vert \X \Vert_{F} = \sqrt{\operatorname{tr}(\X^{\top}\X)}$ denotes the Frobenius norm.

\subsection{Main Difficulties}
A central challenge in the design of high-order geometry-preserving numerical schemes is to simultaneously enforce the nonlinear manifold constraints and achieve a strong convergence order comparable to that of the corresponding Euclidean methods. The following two representative examples are presented to illustrate this tension.

\subsubsection{Euclidean Milstein: Strong Order without Geometry} In Euclidean spaces $\mathbb R^{d}$, the Euclidean Milstein method achieves $1$ SCR for nonlinear SDEs by adding the $3/2$-order It\^o-Taylor term to the Euler-Maruyama increment~\cite{mil1975approximate}. With step size $\delta$ and Wiener increments $\Dms W_j(t) = W_j(t) - W_{j}(m \delta)$, the update rule of the Euclidean Milstein scheme for~\cref{eq:Ito-SDE-sec3} reads
\begin{align*}
    &\hatX(t) =   \hatX_{\m} + \B_I(\hatX_{\m})\Dms t + \half \nm \Sigma_N(\hatX_m)\\ 
                &  + \sum_{j=1}^{d} \B_{j}(\hatX_{\m})\Dms W_{j}(t) + \sum_{j,j'=1}^d \Bigl(\DerNjjp{\hatX_m}\Bigr) \tIjjpms(t), \quad m\delta \le t \le (m+1)\delta,
\end{align*}
where $\tilde I_{jj'}^{m}(t)$ is a proper approximate of the the double It\^o integrals $\tIjjpms(t)$. 

Although the Euclidean Milstein scheme preserves the SCR $1$ when applied to Lie SDEs, two ingredients make the algorithm violate the geometric constraint. First, the pinning drift $\nm\Sigma_N(\hatX_m)$ acts in the normal space. Second, the Milstein increment $\DerNjjp{\hatX_m}$ contains a second-fundamental-form component $S^{G}(\B_{j},\B_{j'})$ that is also normal. Consequently, the normal contributions push the numerical solution away from the manifold, making the iterates leave the Lie group on short time scales.

\subsubsection{Magnus-based Integrators: Geometric without Strong Order}\label{sec:magnus_difficult}

A widely used strategy for preserving group structure is to use the Magnus expansion that pulls the SDE back to the Lie algebra. Writing the state as $X(t)=X_0\exp\bigl(\Omega(t)\bigr)$ with
$\Omega(t)\in\mathfrak g$ and $\exp$ is the matrix exponential, the SDE~\cref{eq:Ito-SDE-sec3} is equivalent to the
$\mathfrak g$-valued SDE
\[
  \mathrm d\Omega
  = A(\Omega)\mathrm dt
    +\sum_{j=1}^{d}\Gamma_j(\Omega)\mathrm dW_j,\qquad
  \Gamma_j=\operatorname{dexp}^{-1}_{-\Omega}(\tg B_j).
\]
Here, $A(\Omega)$ is a drifting matrix that depends on the drift directions $\B_I$ and $\B_j$. The operator $\operatorname{dexp}_{-\Omega}^{-1}$ denotes the inverse differential of the exponential map
\begin{align}\label{eq:magnus}
    \operatorname{dexp}_{-\Omega}^{-1}(\tg B_j)= \sum_{k=0}^{\infty}\frac{\beta_{k}}{k!}
    \bigl(\operatorname{ad}_{\Omega}\bigr)^{k}(\tg B_j) \qquad \operatorname{ad}_{\Omega}(\tg B_j):=[\Omega,\tg B_j],
\end{align}
with $\beta_k$ the Bernoulli numbers. In practice one often approximates the $\mathfrak g$-valued SDE by truncating the infinite series of $\Gamma_j$. Such a truncation, however, is incompatible with the $3/2$-order It\^o-Taylor expansion, as the full infinite series $\Der{\nabla}{\Gamma_j}{\Gamma_{j'}}{\Omega}$ appears simultaneously inside the differentiation operator and in the term being differentiated, so finite truncation may undermines the accuracy of the scheme. Consequently, one must keep the entire series when implementing the Milstein method, which is generally intractable.

Taken together, these two examples reveal the core challenge in constructing a geometry-preserving Milstein scheme over Lie groups. To achieve SCR $1$, one must retain all $3/2$-order It\^o–Taylor contributions, yet to remain on the Lie group these contributions must be decomposed into tangent and normal components and the latter must be treated properly. The challenge makes the design of a geometry-preserving Milstein-type method on $\SOn$ a nontrivial problem.

\section{TaSP-CM Scheme on SO(n)}\label{sec:tasp-so}
To overcome the challenge, we propose a geometric-preserving Milstein method on $\SOn$ through the tangent‑space parametrization (TaSP) framework. In this section, we first introduce the TaSP methodology and then propose our TaSP-CM scheme.

\subsection{Tangent Space Parametrization on SO(n)}
For $\mathbf R_0\in\SOn$, the TaSP framework represents any local solution of~\cref{eq:I-SDE-son} as
\begin{align}\label{eq:tasp-son}
    \mathbf R(t) =\mathbf R_0 +\mathbf R_0 \mathcal Z(t) +\mathbf R_0 \mathsf C\bigl(\mathcal Z(t)\bigr).
\end{align}
In this parametrization, $\mathcal Z(t)\in\son$ is a stochastic tangent displacement to be determined that evolves in the Lie algebra.  The sum $\mathbf R_0+\mathbf R_0\mathcal Z(t)$ therefore gives a first‑order approximation to the true trajectory $\mathbf R(t)$. However, this linear term generally drifts off the manifold.  To maintain the geometric constraint $\mathbf R(t)^{\top}\mathbf R(t)=I$, a normal adjustment $\mathbf R_0 \mathsf C\bigl(\mathcal Z(t)\bigr) \in N_{\bfR}\SOn$ that completely determined by $\mathcal Z(t)$ is introduced to pull $\bf R(t)$ back onto $\SOn$.

Consequently, once the dependence $\mathsf C(\mathcal Z)$ is known and $\mathcal Z(t)$ is approximated with sufficient accuracy, we obtain a reliable approximation $\mathbf R(t)$ whose trajectory preserves the geometry. Fortunately, the results in~\cite{wang2025tangent} give both an explicit formula for the adjustment term $\mathsf C(\mathcal Z)$ and a stochastic dynamic of the tangent displacement $\mathcal Z(t)$.  We summarize these facts in the following lemmas.

\begin{lemma}[\cite{wang2025tangent}]\label{lemma:correction}
Let $\bfR_0\in\SOn$ and $\mathcal Z\in\son$.  If the matrix $I-\mathcal Z^{\top}\mathcal Z$ is positive definite, then the symmetric adjustment $\mathsf C\in\sonp$ that guarantees~\cref{eq:tasp-son} is given by $\mathsf C(\mathcal Z)=\sqrt{ I-\mathcal Z^{\top}\mathcal Z}-I$. Thw term $\sqrt{A}$ denotes the unique positive‑definite square root of the positive‑definite matrix $A$. Based on the adjustment $\mathsf C(\mathcal Z)$, the stochastic tangent displacement $\mathcal Z(t)$ associated with the local solution $\mathbf R(t)$ of the SDE~\cref{eq:I-SDE-son} satisfies the following Euclidean It\^o SDE
\begin{align}\label{eq:sde-z}
   d\mathcal Z =\mathcal B_I\bigl(\mathbf R\bigr) dt +\sum_{j=1}^{d}\mathcal B_j\bigl(\mathbf R\bigr)dW_j, 
\end{align}
on the linear space $\son$.
\end{lemma}

\subsection{Algorithm Design}
Building on \cref{lemma:correction}, we propose our TaSP-CM scheme as follows. Let the step size $\delta$ and Wiener increment $\Dms W_{j}(t) = W_j\bigl( t \bigr) - W_j(m\delta)$. The TaSP-CM scheme first computes a corrected Milstein increment in the lie algebra $\son$ for $m\delta \le t \le (m+1)\delta$,
\begin{align}\label{eq:milZ}
    \hatZ(t) = \B_I(\hatR_{\m})\Dms t + \sum_{j=1}^d \B_j(\hatR_{\m}) \Dms W_{j}(t) + \sum_{j,j'=1}^d  \DerDOjjp{\hatR_m} \Ijjpms(t).  \nonumber 
\end{align}
Here, we replace the classical Euclidean derivative in Milstein term, $\DerNjjp{\hatR_m}$, by the intrinsic CovD $\DerDOjjp{\hatR_m}$ for two reasons. First, the Euclidean derivative $\DerNjjp{\hatR_m}$ introduces a normal component in $\sonp$ that pushes the update out of the Lie algebra $\mathfrak{so}(n)$. Second, the CovD includes the Lie-bracket correction $[\B_j,\B_{j'}]$, which captures the non-commutativity of the group and therefore improves the accuracy of the scheme.

Next we update the trajectory $\bfR(t): [m\delta,(m+1)\delta] \to \SOn $ with the correction term $\nmC(\Z)$, i.e.,
\begin{align*}
    \hatR(t) = \hatR_m + \bfR_m\hatZ(t) + \hatR_m\nmC(\hatZ(t)) = \hatR_m(\hatZ(t) + \sqrt{I - \hatZ^{\top}(t) \hatZ(t)}).
\end{align*}
Finally, connecting the piecewise trajectories yields a geometry‑preserving path of the SDE~\cref{eq:I-SDE-son}. We summarize the TaSP-CM scheme as in \cref{alg:tasp-ms}.

\begin{algorithm}[t]
    \caption{TaSP-CM Scheme on $\SOn$}
    \label{alg:tasp-ms}
    \begin{algorithmic}[1] 
        \State \textbf{Input:} Time step size $\delta$, final time $T$, initial state $\bfR_0 \in \SOn$.
        \State Divide the interval $[0, T]$ into $M$ sub-intervals with length $\delta$.
        \For{$m = 0, \dots, M$}
            \State Compute $\hBIms = \B_I(\hatR_{\m}), \hBjms = \B_j(\hatR_{\m})$ for all $j\in [d]$;
            \State Compute $\hDjjpms =\DerDOjjp{\hatR_m}$ for all $j,j'\in [d]$;
            \State Sample the path $\Dms W_j(t) \sim \mathcal{N}(0,\Dms t)$ for all $j\in [d]$; \label{alg:step:sample}
            \State Approximate the double It\^o integral $\tIjjpms(t) \approx \Ijjpms(t) $ for all $j,j\in [d]$;
            \State Generate $\Z(t)$ for $ m\delta < t \le (m+1)\delta$
               \[  \hatZ(t)=  \hBIms  \Dms t + \frac{1}{2} \sum_{j=1}^d \hBjms \Dms W_j(t) + \sum_{j,j'=1}^d \hDjjpms \tIjjpms(t);\]
            \If{$I - \hatZ^{\top}(t)\hatZ(t)$ is not positive definite}
                \State Go to Step \ref{alg:step:sample}; 
            \Else
                \State Generate trajectory $\hatR(t)$
                \[    \hatR(t) =\hatR_{\m} + \hatR_{\m} \sqrt{ I - \hatZ^{\top}(t)\hatZ(t)}, \quad m\delta < t \le (m+1)\delta;\]
            \EndIf
        \EndFor
        \State \textbf{Output:} Trajectory $\hatR(t)$. 
    \end{algorithmic}
\end{algorithm}

\section{Convergence Analysis on SO(n)} \label{sec:conv-so}
In this section, we analyze the SCR of the proposed TaSP-CM scheme. We start from a simple case where the SDE~\cref{eq:I-SDE-son} is driven by commutative noise, i.e., 
\begin{align*}
    \D_{\B_j}(\B_{j'}(\bfR)) = \D_{\B_{j'}}(\B_{j}(\bfR)), \quad \forall j,j'\in [d], \bfR\in \SOn.
\end{align*} 
We then extend the analysis to the general, non-commutative setting. To carry out our analysis, we impose the following assumptions, which are standard for SDE numerical schemes on both Euclidean spaces and Lie groups~\cite{mil1975approximate,misawa2001lie,wang2025tangent}.

\begin{assumption}\label{asm:drift}
The drift term $\B_{I}(\mathbf R)$ is continuously differentiable and both the function itself and its first derivative are uniformly bounded, i.e., there exist constants $C_{0}, C_{1}>0$ such that
\[
   \|\B_{I}(\mathbf R)\|_F\le C_0, \quad \|\nabla \B_{I}(\mathbf R)\|_F\le C_1 \quad\text{for all } \mathbf R\in\SOn.
\]
\end{assumption}

\begin{assumption}\label{asm:diffusion}
For each diffusion vector field $\B_{j}(\mathbf R),j \in[d]$, we assume twice‐continuous differentiability with uniformly bounded zeroth-, first-, and second-order derivatives, i.e., there exist constants $C_0,C_1,C_2>0$ such that
\[
   \|\B_{j}(\mathbf R)\|_F\le C'_0,\quad
   \|\nabla \B_{j}(\mathbf R)\|_F\le C'_1,\quad
   \|\nabla^{2} \B_{j}(\mathbf R)\|_F\le C'_2
   \quad\text{for all }\mathbf R\in\SOn.
\]
\end{assumption}
\Cref{asm:diffusion} implies that the Euclidean derivative $\Der{\nabla}{\bfR\B_j}{\bfR\B_{j'}}{\bfR}$ and the CovD $\DerDOjjp{\bfR}$ has uniformly bounded zeroth- and first-order derivative, by the bounds $C'_1$ and $C'_2$, respectively.

\subsection{Convergence Analysis with Commutative Noise} We extend the classical commutative-noise assumption on the diffusion fields $\B_j$ from the Euclidean setting to $\SOn$.
\begin{assumption}\label{asm:commutative}
    The SDE~\cref{eq:I-SDE-son} has commutative noise, i.e.,
    \begin{align*}
        \DerDOjjp{\bfR} = \DerDOjpj{\bfR}, \quad \bfR\in \SOn, \quad \forall j,j\in [d].
    \end{align*}
\end{assumption}

Under \cref{asm:commutative}, the double It\^o integral satisfies~\cite{higham2021introduction},
\begin{align*}
    \Ijjpms(t) + \Ijpjms(t) = \Dms W_j(t)\Dms W_{j'}(t) - \eta_{j,j'}\Dms t,
\end{align*}
where the $ \eta_{j,j'}$ is the Kronecker symbol
\begin{align*}
    \eta_{j,j'} = \begin{cases}
        1, \quad j=j', \\
        0, \quad j \neq j'.
    \end{cases}
\end{align*}
Therefore, the TaSP–CM scheme can be simplified without the need for explicitly estimating the double It\^o integral, i.e.,
\begin{align*}
    \hatZ(t) = & \B_I(\hatR_{\m})\Dms t + \sum_{j=1}^d \B_j(\hatR_{\m}) \Dms W_{j}(t)  \nonumber \\
               & \quad + \half \sum_{j\neq j'} \hDjjpms  \Dms W_j(t)\Dms W_{j'}(t) + \half \sum_{j=1}^n \hDjjms\left( \Dms W_{j}^2(t) - \Dms t \right),  
\end{align*}

We now handle the convergence analysis of the TaSP-CM scheme in the commutative noise setting. First, we introduce an frozen‑field Milstein series
\begin{align}\label{eq:frozen-mil}
     \bfR_\delta(t) = & \bfR_0 + \sum_{m=1}^{\sumub} \hRms \hBIms \Dms t +  \sum_{m=1}^{\sumub}  \hRms\sum_{j=1}^d  \SS(\hBjms, \hBjms) \Dms t  \\ 
   & + \sum_{m=1}^{\sumub}\hRms\sum_{j=1}^d \hBjms \Dms W_j(t)  + \sum_{m=1}^{\sumub}\hRms\sum_{j,j'=1}^d \hnjjpms \Ijjpms \nonumber \\
    \hnjjpms = & \Der{\nabla}{\hRms\B_j}{\hRms\B_{j'}}{\hatR_m}.  \nonumber
\end{align}
The frozen‑field Milstein series $\hatR_{\delta}(t)$ is actually the Euclidean Milstein recursion, where at each time step $m$ the drift $\hatB_I$ and the diffusion fields $\hatB_j$ are evaluated at the TaSP-CM state $\hatR(m\delta)$.  In this cases, the error of the TaSP-CM scheme is decomposed as
\begin{align*}
        &  \E\big[ \sup_{0 \le t\le T} \| \bfR(t) - \hatR(t) \|^2_F \big] \\ 
    \le & 2\E\big[ \sup_{0 \le t\le T}\| \hatR(t)- \bfR_\delta(t) \|_F^2 \big] + 2\E\big[ \sup_{0 \le t\le T} \|  \bfR_\delta(t) - \bfR(t)\|^2_F \big] 
\end{align*}

Second, We analyze the two terms on the right‑hand side separately.  Because the drift and diffusion fields are evaluated at the same state in both the TaSP and the frozen‑field Milstein series $\bfR_\delta$, and we notice that $\POtg(\hnjjpms(t)) = \hDjjpms(t)$, the trajectories $\hatR(t)$ and $\hatR_\delta(t)$ have identical tangent increment. Consequently, the error between $\hatR(t)$ and $\hatR_\delta(t)$ solely depends on their normal increment. The normal increment of $\hatR(t)$ is $\sum_{m=1}^{\sumub} \hatR_{\m} \nmC_{\m}$, whereas that of $\hatR_\delta(t)$ comprises sum of second‑order terms $S(\hBjms, \hBjpms)$. Concretely, we have
\begin{align*}
     \hatR(t) - \bfR_\delta(t) = & \sum_{m=1}^{\sumub}  \hatR_{\m}\Bigl(\nmC_{\m} -   \half \sum_{j=1}^d \SS(\Z_j, \Z_j) \Bigr) +\text{higher-order terms}
\end{align*}
From~\cite{wang2025tangent} we have the estimate
\begin{align*}
\|\nmC(\Z)-\half\SS(\Z,\Z)\bigr\|_{F} \le \half\|\Z\|_F^{4},
\end{align*}
which implies that there is an $\|\tilde Z\|_F^{8}=\calO(\delta^{4})$ squared error between $\hatR(t)$ and $\hatR_\delta(t)$ in local time step $[m\delta,(m+1)\delta]$. Summing the error over $m$ yields an $\mathcal O(\delta^{2})$ mean‑square error, as stated in the following lemma.
\begin{lemma}\label{lemma:1} Under \cref{asm:drift,asm:diffusion,asm:commutative}, there holds
\begin{align*}
    \bfE\big[ \sup_{0 \le t\le T}\| \hatR(t)- \bfR_\delta(t) \|_F^2 \big] \le \calO(\delta^2).
\end{align*}
\end{lemma}
The proof of~\cref{lemma:1} is left in the~\cref{app:lemma1}. \Cref{lemma:1} reveals that each iterate of the TaSP–CM scheme can be viewed as a classical Euclidean Milstein step plus a fourth‑order normal‑space adjustment. This ``Euclidean Milstein+geometric patch" perspective is the key mechanism that allows us to extend the geometric-preserving Milstein method over $\SOn$.

Then we turn to the term $\bfE\big[ \sup_{0 \le t\le T} \|  \bfR_\delta(t) - \bfR(t)\|^2_F \big] $. By truncating the  It\^o-Taylor expansion of the real solution $\bfR(t)$ at order $3/2$ and applying the Gr\"onwall inequality, we obtain the following lemma.

\begin{lemma}\label{lemma:2} Under \cref{asm:drift,asm:diffusion,asm:commutative}, there holds
\begin{align*}
    \E\big[ \sup_{0 \le t\le T} \| \hatR_\delta(t) - \bfR(t) \|^2_F \big]  = \calO(\delta^2).
\end{align*}
\end{lemma}
The proof of \cref{lemma:1} is left in the~\cref{app:lemma2}. Finally, we combine \cref{lemma:1,lemma:2} and obtain the following SCR.
\begin{theorem}\label{thm:convergence-comm} 
    Suppose \cref{asm:drift,asm:diffusion,asm:commutative} hold, and the double It\^o integral estimator is given as 
    \begin{align*}
        \tIjjpms(t) = \half \Dms W_j(t) \cdot \Dms W_{j'}(t) - \eta_{j,j'}\Dms t.
    \end{align*}
    The TaSP-CM scheme applied to \eqref{eq:I-SDE-son} attains SCR $1$, i.e., with time step $\delta$, the strong error satisfies
\begin{align*}
          \E\big[ \sup_{0 \le t\le T} \| \bfR(t) - \hatR(t) \|^2_F \big]^{\half} = \calO(\delta).
     \end{align*}
\end{theorem}
\begin{proof}
It follows directly by \cref{lemma:1,lemma:2}.
\end{proof}

\Cref{thm:convergence-comm} introduces a Milstein scheme that attains a SCR of 1 while exactly preserving orthogonality on $\mathrm{SO}(n)$ in the commutative-noise setting. The result bridges the gap between the classical Euclidean Milstein method and previous computationally tractable GPNS, whose SCR were $1/2$. The TaSP-CM scheme therefore offers a practical tool for problems where both geometric exactness and high numerical accuracy are essential.



\subsection{Convergence Analysis with Non-Commutative Noise}
Unlike SDEs in Euclidean spaces, where typical noise structures (e.g., constant or diagonal noise) are commutative \cite{higham2021introduction}, most SDEs on $\mathrm{SO}(n)$ involve non-commutative noise, due to the non-commutative structure of the Lie algebra.  In particular, the canonical Brownian motion on $\mathrm{SO}(n)$ is non-commutative.

\begin{example}[Brownian motion on $\SOn$]\label{ex:bw-son}
Let $\{\tg E_{jk}\}_{1\le j<k\le n}$ be the standard basis of the Lie algebra $\son$, defiend by
\begin{align}\label{eq:std-b-son}
    (\tg E_{jk})_{pq}=\eta_{j,p}\eta_{k,q}-\eta_{j,q}\eta_{k,p}.
\end{align}
The Brownian motion on $\SOn$~\cite{hsu2002stochastic} satisfies
\[
  d\bfR= \bfR (\frac{n-1}{2})Idt + \bfR\sum_{1\le j<k\le n}\tg E_{jk}(\bfR) \mathrm dW_{jk}(t).
\]

To see that the diffusion directions are non-commutative, fix any triple of distinct indices $j < k < l$.  Because $\mathfrak{so}(n)$ is non-abelian for $n\ge 3$, we have~\cite{gallier2020differential} 
\[
   \Der{\DO}{\B_{jk}}{\B_{kl}}{\hatR} = 0 + [\tg E_{jk},\tg E_{kl}]=  \tg E_{jl} \neq - \tg E_{jl} = \Der{\DO}{\B_{kl}}{\B_{jk}}{\hatR}.
\]
which illustrates that the Brownian motion on $\SOn$ is inherently non-commutative.
\end{example}

The non-commutative noise above means that we must work in the setting where $\DerDOjjp{\bfR} \neq \DerDOjpj{\bfR}$ and give accurate approximation of the double It\^o integral $\Ijjpms$. Following~\cite{kastner2023analysis}, we simulate $\Ijjpms$ using $2h$ independent Gaussian random variables. The construction is summarized in \cref{alg:AppI}.
\begin{algorithm}[t]
    \caption{Approximate $\Ijjpms(t)$~\cite{kastner2023analysis}}
    \label{alg:AppI}
    \begin{algorithmic}[1]
        \State \textbf{Input:} Time $t>0$, truncation level $h\in\mathbb N$, indices $j,j'\in\{1,\dots,d\}$, Wiener increment $\Dms W_j(t),\Dms W_{j'}(t)$
        
        \State For each $r=1,\dots,h$, draw $a_{j,r},\ b_{j,r},a_{j',r},\ b_{j',r}\sim \mathcal N(0,1)$ i.i.d.;
        \State Draw i.i.d. $\gamma_{j,h} , \gamma_{j',h}\sim \mathcal N(0,1)$;
        \State Compute $\rho_h = \frac{1}{2\pi^2}\left(\frac{\pi^2}{6}-\sum_{r=1}^{h}\frac{1}{r^2}\right)$;
        
        \State Compute
        \begin{align*}
             R_{j,j'}^h &= \sqrt{\frac{\rho_h }{2\pi^2}} \Big( \Dms W_j \gamma_{j',h} - \gamma_{j,h} \Dms W_{j'}\Big), \\
             \tilde b_{j,r} &= b_{j,r} - \sqrt{\tfrac{2}{\Dms t}}\Dms W_j, \qquad \tilde b_{j',r} = b_{j',r} - \sqrt{\tfrac{2}{\Dms t}}\Dms W_{j'}, \\
             A^{p,h}_{j,j'} &= \frac{1}{2\pi}\sum_{r=1}^{h}\frac{1}{r}\Big(a_{j,r}\tilde b_{j',r} -a_{j',r}\tilde b_{j,r}\Big); 
        \end{align*}
        
        \State \textbf{Return:} $K_{j,j'}^h = \dfrac{1}{2}\Dms W_j(t)\cdot \Dms W_{j'}(t) - \dfrac{1}{2} \eta_{j,j'}\Dms t + R_{j,j'}^h \sqrt{\Dms t} + A^{p}_{j_1,j_2}  \Dms t$.
    \end{algorithmic}
\end{algorithm}

\begin{lemma}[\cite{kastner2023analysis}]\label{lem:DI1-DI2}
Let $h\in \mathbb N$, $\tIjjpms = K_{j,j'}^h$ as implemented by~\cref{alg:AppI}. The approximation error satisfies,
\begin{align*}
    \DIjjpms := \Ijjpms - \tIjjpms = \frac{1}{2\pi}\sum_{r=h+1}^{\infty}\frac{1}{r}\Big(a_{j,r}b_{j',r}-a_{j,r}b_{j',r}\Big) \Dms t,  \quad j \neq j'\in[d],
\end{align*}
with $\displaystyle\E[\DIjjpms]=0$ and $\displaystyle \E[(\DIjjpms)^2]=\frac{\delta^2}{2\pi^2 h^2}$.
\end{lemma}

With the approximation in hand, we again establish the SCR $1$ result on non-commutative noise setting for the TaSP-CM schemes.
\begin{theorem}[Strong convergence under non-commutative noise]\label{thm:convergence-non-comm}
Assume that \cref{asm:drift,asm:diffusion} hold and let $\delta$ denote the time-step size.  
Choose the double It\^o integral approximation $\tIjjpms = K_{j,j'}^{h}$ with $h \ge \delta^{-\half}$ as implemented in \cref{alg:AppI}.  
Then the TaSP-CM scheme applied to \eqref{eq:I-SDE-son} achieves strong order $1$, i.e.,
\begin{align*}
    \E\bigl[\sup_{0 \le t \le T} \| \bfR(t) - \hatR(t) \|_F^{2}\bigr]^{\half} = \calO(\delta).
\end{align*}
\end{theorem}

The proof of~\cref{thm:convergence-non-comm} is left in~\cref{app:non-commute}. \Cref{thm:convergence-non-comm} shows that the TaSP–CM scheme achieves SCR $1$ under the non-commutative noise. The requirement $h \ge \delta^{-\half}$, which is exactly the same condition used in the Euclidean Milstein method~\cite{mil1975approximate}, guarantees that the mean-square error of the double It\^o integral approximation remains $\mathcal{O}(\delta^2)$. Consequently, the scheme raises the available SCR for GPNS scheme on $\SOn$ from $1/2$ to 1 in the non-commutative setting.

In the commutative case, the double It\^o integral reduces to a simple product of Wiener increments and can be computed exactly without any additional sampling. In contrast, the non-commutative setting requires a sampling cost of $\mathcal{O}(M h d(d-1))$ independent Gaussian variables, where $M = T/\delta$ is the number of time steps. Since $M=\mathcal{O}(\delta^{-1})$ and $h=\mathcal{O}(\delta^{-1/2})$, the overall
computational cost scales as $\mathcal{O}(\delta^{-3/2})$ Gaussian calls. 

Therefore, the TaSP--CM scheme attains a Gaussian-call complexity of $\mathcal{O}(\delta^{-3/2})$, matching the Euclidean Milstein method, while improving over the geometric Euler--Maruyama method, which requires $\mathcal{O}(\delta^{-2})$ Gaussian calls to achieve the same accuracy.

\section{TaSP-CM over SE(n)}\label{sec:tasp-se}
In this section, we extend the TaSP–CM scheme from the special orthogonal group $\SOn$ to the special Euclidean group
$\SEn = \SOn \ltimes \R^{n}$. Leveraging the semidirect-product structure, \cref{thm:tasp-se} establishes the TaSP framework for $\SEn$.

\begin{algorithm}[t]
    \caption{TaSP-CM Scheme on $\protect\SEn$}
    \label{alg:tasp-ms-sen}
    \begin{algorithmic}[1]
        \Require Time step size $\delta$, final time $T$, initial state $\bfE_0 \in \SOn$.
        \State Divide the interval $[0, T]$ into $M$ sub-intervals with length $\delta$.
        \For{$m = 0, \dots, M$}
            \State Compute $(\hBIRms, \hBIPms)= \B_I(\hatE_{\m})$ for all $j\in [d]$;
            \State Compute $(\hBjRms,\hBjPms) = \B_j(\hatE_{\m})$ for all $j\in [d]$;
            \State Compute $\hDjjpRms = \nabla_{\bfE \tg V_j}{\hBjpRms} + \thalf [\hBjRms,\hBjpRms]$ for all $j,j'\in [d]$;
            \State Compute $\hDjjpPms =\nabla_{\bfE \tg V_j}{\hBjpPms} + \thalf \hBjRms\hBjpPms$ for all $j,j'\in [d]$;
            
            \State Sample the path $\Dms W_j(t) \sim \mathcal{N}(0,\Dms t)$ for all $j\in [d]$; \label{alg:sen:sample}
            \State Approximate the double It\^o integral $\tIjjpms(t) \approx \Ijjpms(t) $ for all $j,j\in [d]$;
            
            \State Generate $\big(\hatZ(t),\hatv(t)\big)$ for $ m\delta < t \le (m+1)\delta$
            \[
            \begin{aligned}
                 \hatZ(t) &= \hBIRms \Dms t + \frac{1}{2} \sum_{j=1}^d \hBjRms \Dms W_j(t)  + \sum_{j,j'=1}^d \hDjjpRms \tIjjpms(t), \\
                 \hatv(t) &= \hBIPms \Dms t + \frac{1}{2} \sum_{j=1}^d \hBjPms \Dms W_j(t) + \sum_{j,j'=1}^d \hDjjpPms \tIjjpms(t);  
            \end{aligned}
            \]
            \If{$I - \hatZ^{\top}(t)\hatZ(t)$ is not positive definite}
                \State Go to Step \ref{alg:sen:sample}; 
            \Else
                \State Generate trajectory $\hatE(t)$
                \[
                \begin{cases}
                    \hatR^E(t) =\hatR_{m}^E \sqrt{ I - \hatZ^{\top}(t)\hatZ(t)},  \\
                    \hatp^E(t) =\hatp_{m}^E + \hatR_{m}\hatv(t), \\
                    \hatE(t) = (\hatR^E(t) , \hatp^E(t)) 
                \end{cases} \quad m\delta < t \le (m+1)\delta ;
                \]
            \EndIf
        \EndFor
        \State \textbf{Output:} Trajectory $\hatE(t)$.
    \end{algorithmic}
\end{algorithm}
\begin{theorem}[TaSP adjustment on $\SEn$]
\label{thm:tasp-se}
Let $\bfE_0\in\SEn$ and $\tg V=(\tg Z,v)\in\sen$.
If $I-\tg Z^{\top}\Z$ is positive definite, then $\nm U(\tg V) = \bigl( \sqrt{I-\tg Z^{\top}\tg Z}-I,0\bigr) \in \sen^\perp$ is the normal adjustment for $\bfE=\bfE_0(I+\tg V+\nm U(\tg V))$ lies in $\SEn$. Moreover, the tangent motion $\tg V(t)$ associated with the local solution of~\cref{eq:I-SDE-sen} on $\SEn$ satisfies the Euclidean It\^o SDE
\begin{align}
    \begin{cases}
    d\mathcal Z =\mathcal B_I^{\bfR}\bigl(\bfE\bigr) dt +\sum_{j=1}^{d}\mathcal B_j^{\bfE}\bigl(\mathbf R\bigr)dW_j \in \son, \\
    dv = \B^{\bf p}_{I}\bigl(\bfE\bigr)dt + \sum_{j=1}^{d} \B^{\bf p}_{j}\bigl(\bfE\bigr) dW_{j}\in \R^n.
    \end{cases}
\end{align}
where $B_I=(B_I^{R},B_I^{p})$ and $B_j=(B_j^{R},B_j^{p})$ are the drift and diffusion fields of~\cref{eq:Ito-SDE-sec3}.
\end{theorem}

\Cref{thm:tasp-se} follows the idea of~\cref{lemma:correction}, and thus the proof is  provided in~\cref{app:tasp-se}. \Cref{thm:tasp-se} shows that, thanks to the semidirect-product structure of $\SEn=\SOn\ltimes\R^{n}$, the TaSP construction on $\SEn$ similar to that in $\SOn$. With the TaSP adjustment, the rotational component evolves exactly as in the $\SOn$ case, while the translational component follows an SDE in $\R^{n}$. Using this observation, we extend the TaSP-CM scheme to $\SEn$, as summarized in \cref{alg:tasp-ms-sen}.

To carry the SCR of $1$ from $\SOn$ to $\SEn$, we adapt the regularity assumptions on the drift and diffusion fields (cf. \cref{asm:drift,asm:diffusion}) to the $\SEn$ setting.
\begin{assumption}\label{as:reg-se}
The drift and diffusion fields
$\B_I=(\B_I^{R},\B_I^{p})$ and
$\B_j=(\B_j^{R},\B_j^{p})$ satisfy
\begin{itemize}
    \item [(i)]  $\B_I^{\bfR},\B_I^{\bf p}$ are continuously differentiable with bounded derivatives;
    \item [(ii)] For all $j\in[d]$, $\B_j^{\bfR},\B_j^{\bf p}$ are twice continuously differentiable with all derivatives up to second order bounded.
\end{itemize}
\end{assumption}

The following~\cref{thm:convergence-sen} gives the SCR of the TaSP-CM scheme on $\SEn$ in the both commutative and non-commutative noise setting.

\begin{theorem}[Strong convergence on $\SEn$]\label{thm:convergence-sen}
Assume that \cref{as:reg-se} holds, and let $\delta$ be the time–step size.  
For the double It\^o integral estimator, choose
\begin{align*}
    \tIjjpms =
    \begin{cases}
        \tfrac12\Delta_{m} W_{j}(t)\Delta_{m} W_{j'}(t) - \eta_{j,j'}\Delta_{m}t, & \textit{commutative noise},\\[6pt]
        I_{j,j'}^{h}, \quad h \ge \delta^{-\thalf}, & \textit{non-commutative noise}.
    \end{cases}
\end{align*}
Then the TaSP-CM scheme applied to \eqref{eq:I-SDE-sen} on $\SEn$ achieves strong order~$1$, i.e.,
\[
    \E\bigl[\sup_{0 \le t \le T} \|\bfE(t) - \hatE(t)\|_F^{2}\bigr]^{\half} = \calO(\delta).
\]
\end{theorem}

The proof of \cref{thm:convergence-sen} refers to~\cref{app:conv-sen}. It handles the rotational component exactly as in the $\SOn$ case and then incorporates an additional argument for the translational part. As a result, the strong convergence rate $1$ established on $\SOn$ extends to the full special Euclidean group $\SEn$, which indicates that the TaSP–CM scheme can be deployed in border application such as robotic pose estimation, and autonomous vehicle navigation.

\section{Numerical Simulation}\label{sec:experiments}

In this section, we test TaSP–CM scheme on SDEs over $\mathrm{SO}(n)$ and $\SEn$ within commutative and non-commutative noise via several simulation.   We examine both strong convergence behavior and geometry-preserving ability, and we compare the results with several baseline numerical schemes.
\subsection{Commutative Noise over SO(n)}
In commutative noise cases, we study the following It\^o SDE on $\SOn$,
\begin{align}\label{eq:exp-com-I}
    d\bfR = & \bfR\Big(-\half\sum_{s=1}^{d}(d-s+1) \sin\big(\theta_s(\bfR)\big)\cos\big(\theta_s(\bfR)\big)  \tg J_s\Big)dt \\
    & \qquad + \bfR\Big(\half\sum_{s=1}^{d} (d-s+1)\cos\big(\theta_s\big(\bfR)\big)^2 \tg J^2_s\Big) dt +\sum_{j=1}^{d}\bfR\B_j(\bfR)dW_j,
\end{align}
where each diffusion field $\tg B_j : \mathrm{SO}(n) \to T_R\mathrm{SO}(n)$ is defined as
\begin{align*}
    \tg B_j(\bfR)= \sum_{s=1}^{j}\cos\big(\theta_s(\bfR)\big)\bfR\tg J_s, \quad
    \theta_s(R) := -\half \tr(\tg J_s^T \log \bfR), 
\end{align*}
with $j =1 ,2, \dots\lfloor \tfrac{n}{2} \rfloor$. Here, $\tg J_s \in \mathfrak{so}(n)$ denotes the planar rotation generator in the $(2s-1,2s)$ coordinate plane, defined componentwise by
\begin{align*}
  (\tg J_s)_{pq} = -\eta_{p,2s-1}\eta_{q,2s} + \eta_{p,2s}\eta_{q,2s-1}, \qquad s \in [d].
\end{align*}

A direct calculation shows that
$$\DerDOjjp{\bfR} = \sum_{s=1}^{\min(j,j')} \sin(\theta_s) \bfR \tg J_s = \DerDOjpj{\bfR},$$ and hence the diffusion vector fields commute.

In the following, we conduct our TaSP-CM scheme to simulate~\cref{eq:exp-com-I} in $\mathrm{SO}(n)$ over the interval $[0,0.25]$, where the coefficient and double Intergral as
\begin{align*}
\B_I =  0 \quad \tg B_j(\bfR)= \sum_{s=1}^{j}\cos\theta_s\bfR\tg J_s, \quad \Ijpjms(t) = \half \Dms W_j(t)\Dms W_{j'}(t),
 \quad j,j' \in [d].
\end{align*}
The step size is varied as $\delta = 2^{-k}$ for $k = 7,\ldots,14$.  For each $\delta$ we estimate the strong error from 50 Monte-Carlo realizations, and the error is obtained by comparing with a high-resolution reference trajectory, which is computed by TaSP–EM scheme~\cite{wang2025tangent} with $\delta_{\text{ref}} = 2^{-21}$.  We also compare the proposed TaSP–CM integrator with the TaSP–EM scheme \cite{wang2025tangent}, the classical Eu–Milstein (Eu–Mil) method \cite{mil1975approximate}, and the structure-preserving RKMK method (S-RKMK) of \cite{muniz2023strong}.  In the S-RKMK implementation we employ the R\"o\ss{}ler’s scheme of strong order 1~\cite{rossler2010strong} to the truncate the Magnus expansion at level $p = 0$.

\begin{figure}[t]
    \centering
    \begin{subfigure}[t]{0.45\textwidth}
        \centering
        \includegraphics[width=\linewidth]{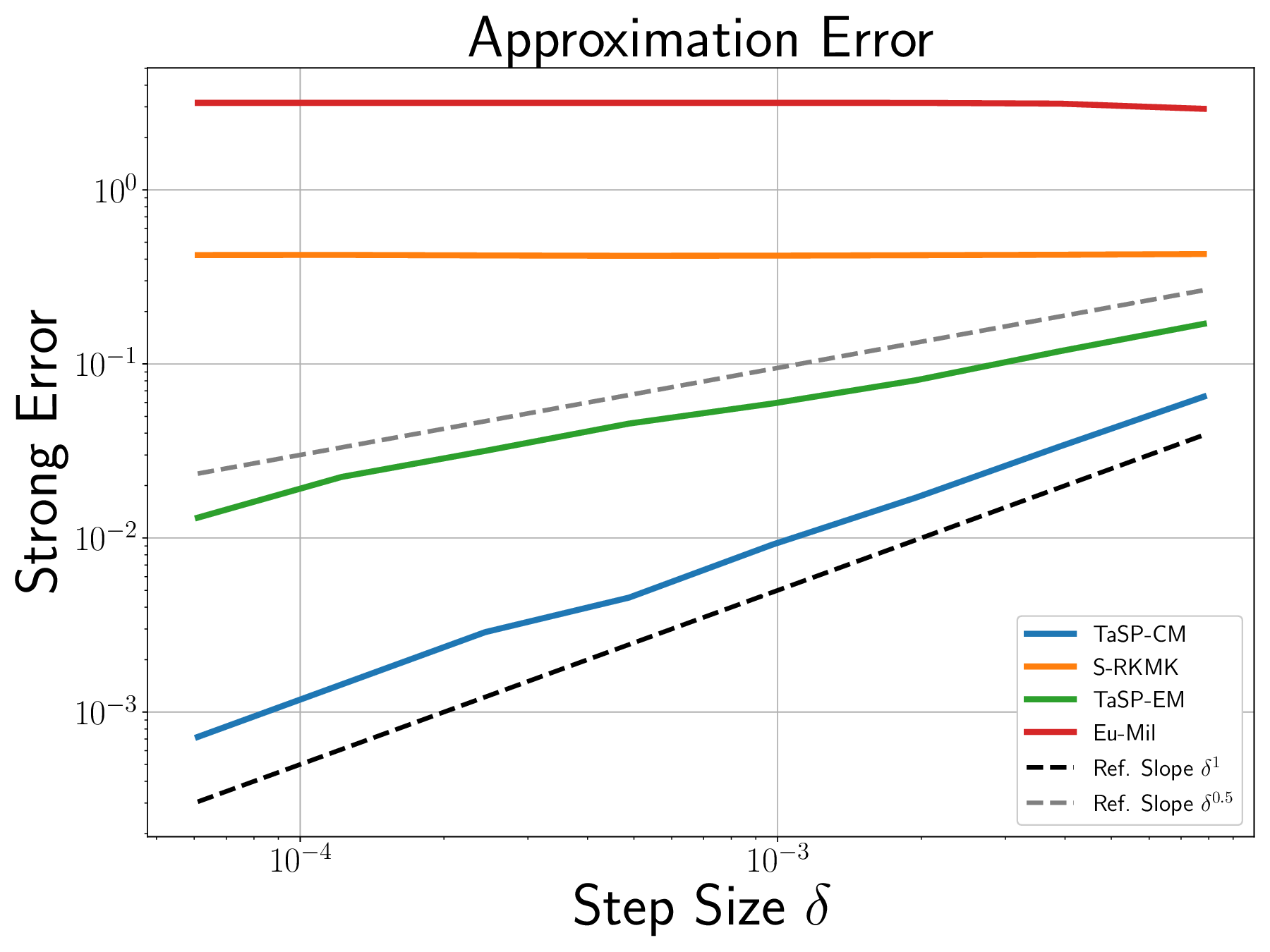}
        \caption{Mean-square error vs. step size $\delta$}
        \label{fig:cartan-err}
    \end{subfigure}
    \hfill
    \begin{subfigure}[t]{0.45\textwidth}
        \centering
        \includegraphics[width=\linewidth]{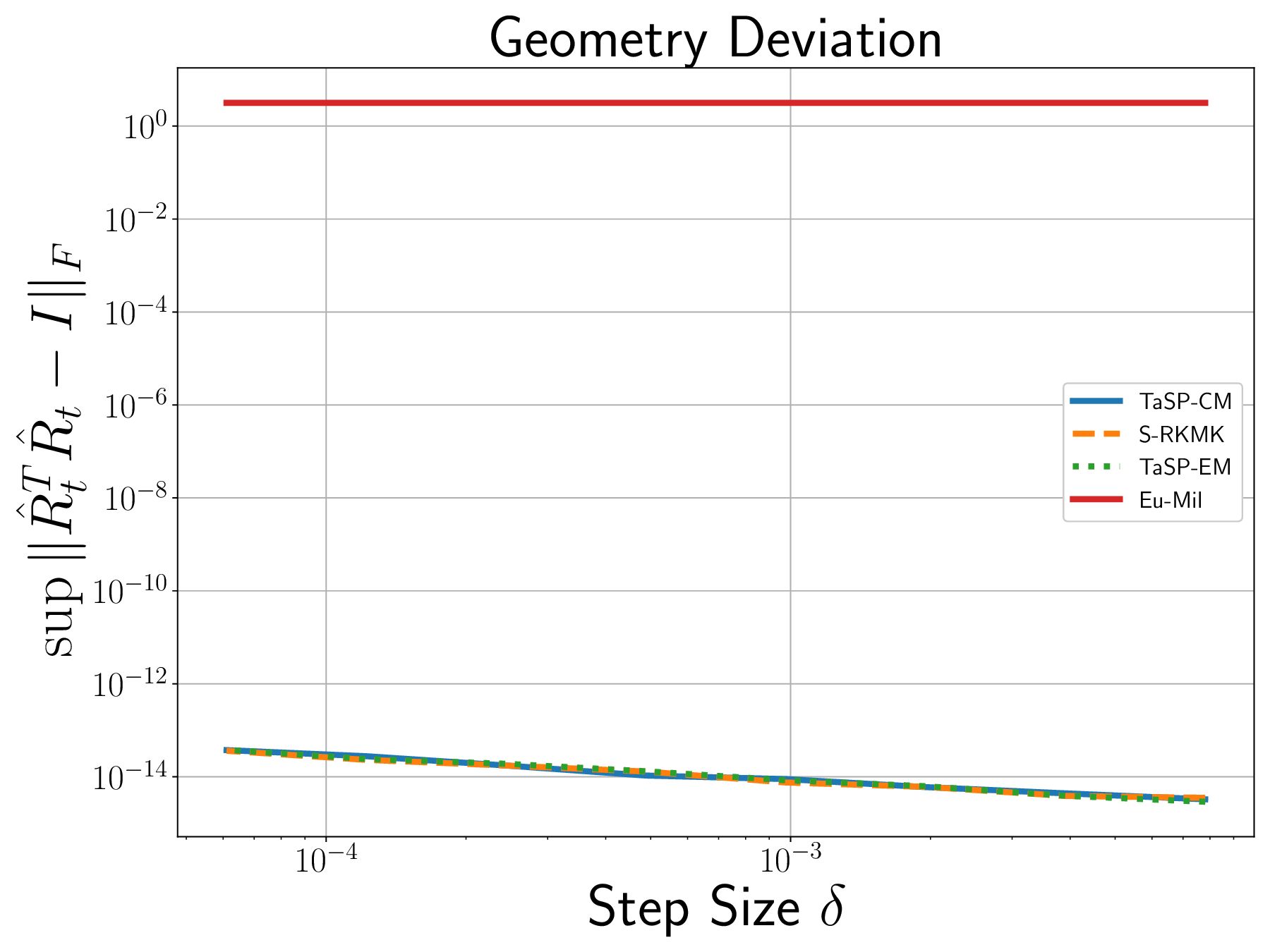}
        \caption{Geometric perseverance vs step size $\delta$}
        \label{fig:cartan-gp}
    \end{subfigure}
    \caption{Numerical performance of TaSP–CM and competing methods under commutative noise over SO(n). Left: convergence rate. Right: geometric preservation.}
    \label{fig:cartan-compare}
\end{figure}
\Cref{fig:cartan-err} plots the strong error $\E\bigl[\max_{0\le t\le T} \bigl\|\mathbf R(t)-\widehat{\mathbf R}(t)\bigr\|_{F}^{2} \bigr]^{\half}$ versus $\delta$ on a log–log scale. From the result, we show that our TaSP–CM scheme attains an empirical SCR of order $1$, matching the theoretical prediction. By contrast, the TaSP–EM method reaches the empirical SCR $1/2$. The S-RKMK solution exhibits an $\mathcal O(10^{-1})$ bias while the Eu–Mil scheme does not converge.

\Cref{fig:cartan-gp} plots the geometric deviation $\sup_{0\le t\le T}\bigl\|\widehat{\mathbf R}^{\top}(t)\widehat{\mathbf R}(t)-I\bigr\|_{F}$ throughout the simulation.  All TaSP-CM, S-RKMK and TaSP–EM maintain orthogonality with errors of $\calO(10^{-13})$, whereas the Eu-Mil method drifts off the manifold. These results match the theoretical analysis and demonstrate the geometric perseverance ability of the TaSP–CM scheme.

\subsection{Manifold Optimization over SO(n)}
We consider the It\^o SDE on $\SOn$
\begin{align}\label{eq:exp-non-son}
   \nonumber d\bfR = & \POtg\big(\nabla a(\bfR)\big)dt 
   + \sum_{1 \le j < k \le n} \Der{\DO}{\sigma(\bfR)\tg E_{jk}}{\sigma(\bfR)\tg E_{jk}}{\bfR}dt \\
   & + \sum_{1 \le j < k \le n}\SSO\big(\sigma(\bfR)\tg E_{jk},\sigma(\bfR)\tg E_{jk}\big)dt 
   + \sum_{1 \le j < k \le n}\sigma(\bfR)\tg E_{jk} dW_{jk}(t),
\end{align}
where $a,\sigma:\SOn\to\R$ are scalar fields. Here $\{\tg E_{jk}\}_{1\le j<k\le n}$ denotes the standard skew-symmetric basis on $\SOn$ defined in~\cref{eq:std-b-son}. Equation~\eqref{eq:exp-non-son} serves as a prototypical model for stochastic gradient descent for manifold optimization on $\SOn$~\cite{yuan2019global,bortoli2022riemannian}. For our experiments we set
\[
a(\bfR)=\tfrac12\|A\bfR-B\|_{\mathrm F}^2,\qquad 
\sigma(\bfR)=\sigma_0 g\Big(1-\frac{1}{n}\tr(\bfR)\Big)M \qquad g(x) = \frac{1}{1+e^{-x}},
\]
with $g(x)$ the sigmoid function and $\sigma_0=\frac{3}{2}$. This choice mimics an orthogonal Procrustes problem~\cite{mil1975approximate} with stronger noise away from the identity. We take $A$ to be diagonal with entries drawn uniformly in $[0.5,1]$, and set 
\[
B=\tfrac1{50}\sum_{i=1}^{50} u_i v_i^\top, \qquad u_i,v_i\in\R^n,
\]
where each $u_i,v_i \in \R^n$ is an independent random unit vector. 

From~\cref{ex:bw-son}, we have that the diffusion directions $\{\tg E_{jk}\}$ do not commute. Therefore, we simulate~\eqref{eq:exp-non-son} using the TaSP-CM scheme in the non-commutative noise setting. We simulate trajectories on $\mathrm{SO}(50)$ over the interval $[0,0.25]$ using step sizes $\delta = 2^{-k}$ for $k = 7,\dots,14$ over 50 independent Monte-Carlo runs..  The error is also obtained by comparing with high-resolution reference path is generated with $\delta_{\text{ref}} = 2^{-18}$ by TaSP–EM scheme.  We also compare our methods with TaSP–EM, Eu-Mil scheme and the S-RKMK schme.

\begin{figure}[t]
    \centering
    \begin{subfigure}[t]{0.45\textwidth}
        \centering
        \includegraphics[width=\linewidth]{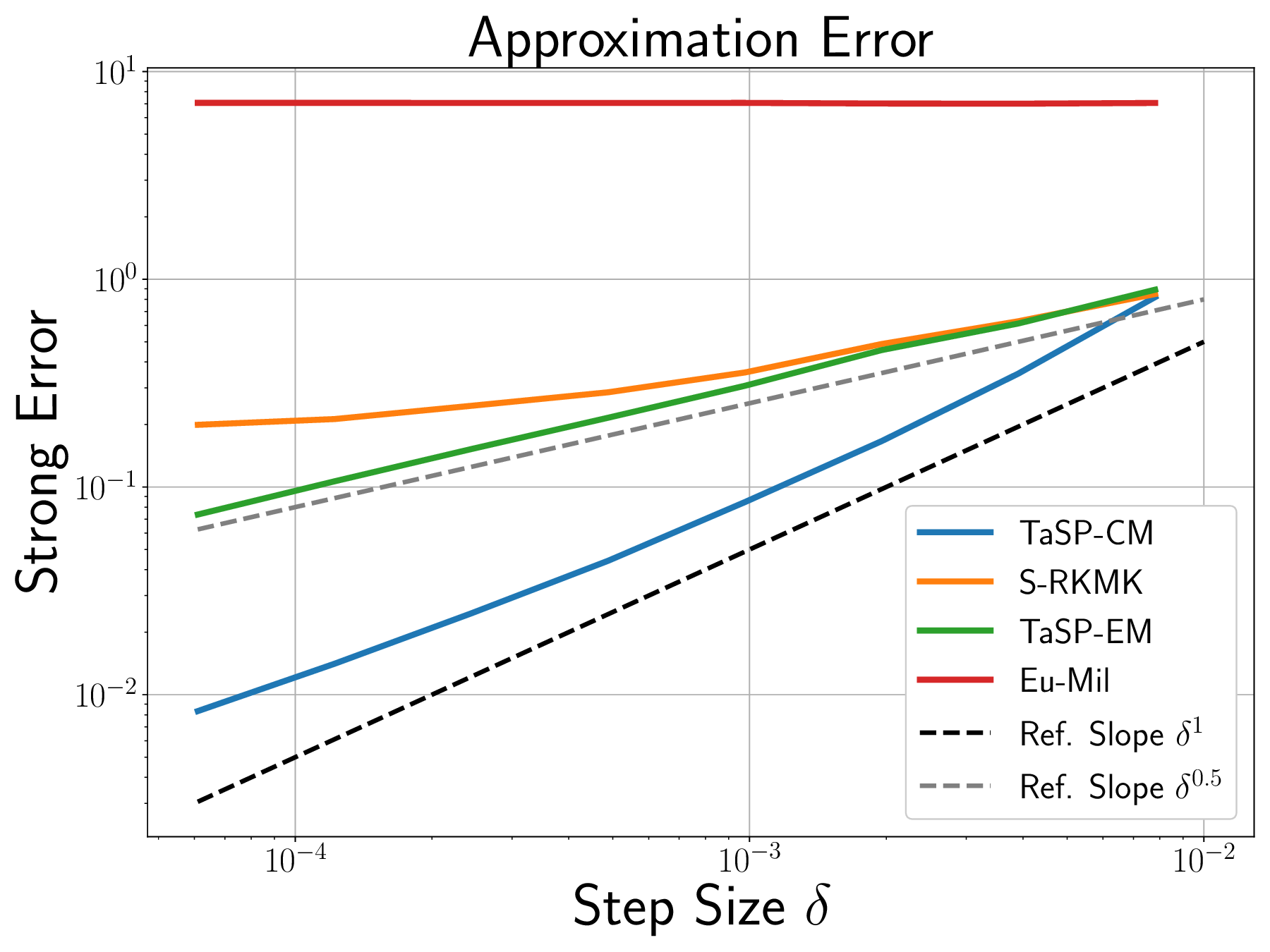}
        \caption{Mean-square error vs.\ step size $\delta$}
        \label{fig:bm-err}
    \end{subfigure}
    \hfill
    \begin{subfigure}[t]{0.45\textwidth}
        \centering
        \includegraphics[width=\linewidth]{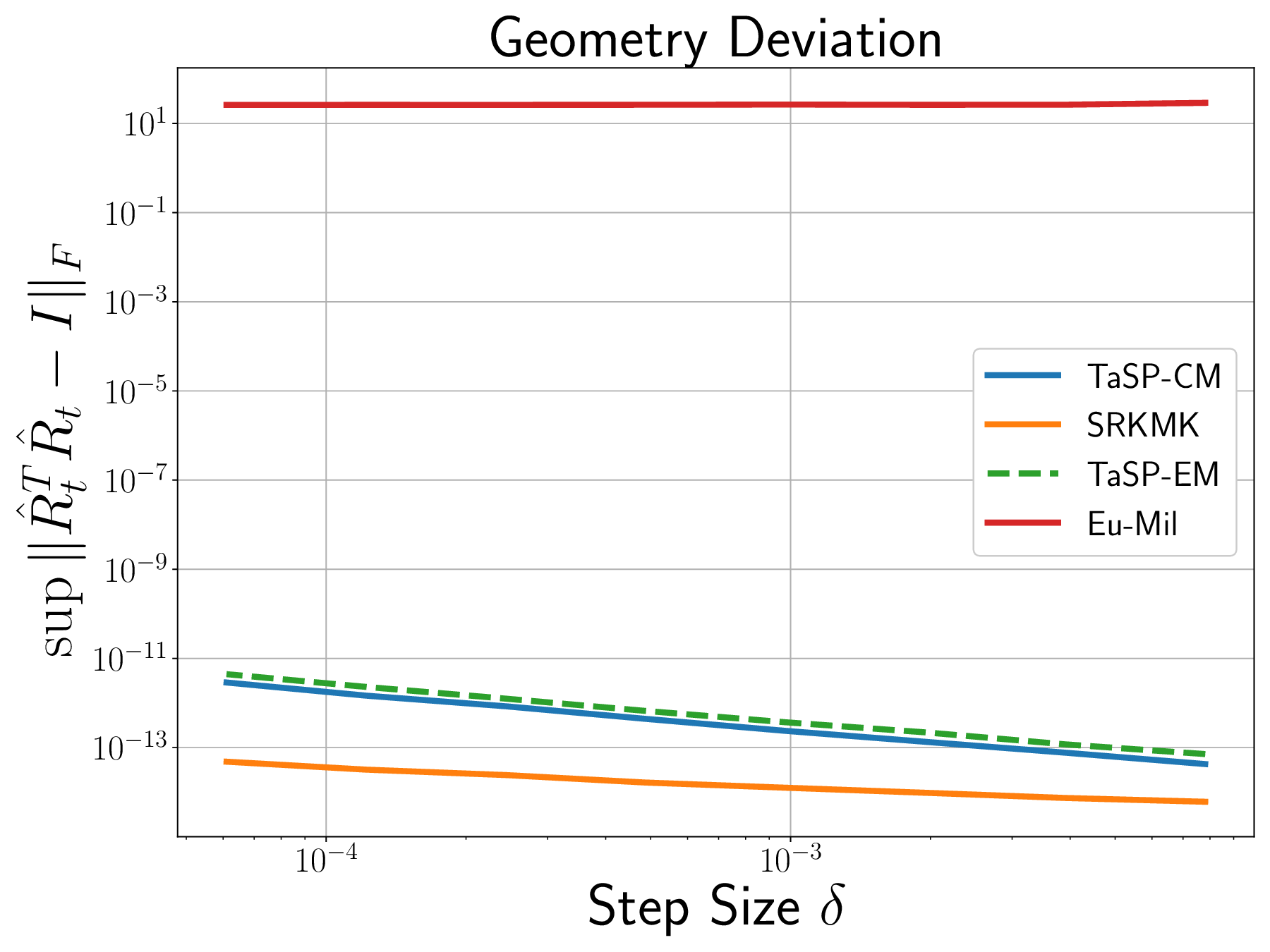}
        \caption{Geometry deviation vs.\ step size $\delta$}
        \label{fig:bm-gp}
    \end{subfigure}
    \caption{Numerical performance of TaSP–CM and competing methods under commutative noise over SO(n). Left: convergence rate. Right: geometric preservation.}
    \label{fig:bm-compare}
\end{figure}

\begin{figure}[t]
    \centering
    \begin{subfigure}[t]{0.45\textwidth}
        \centering
        \includegraphics[width=\linewidth]{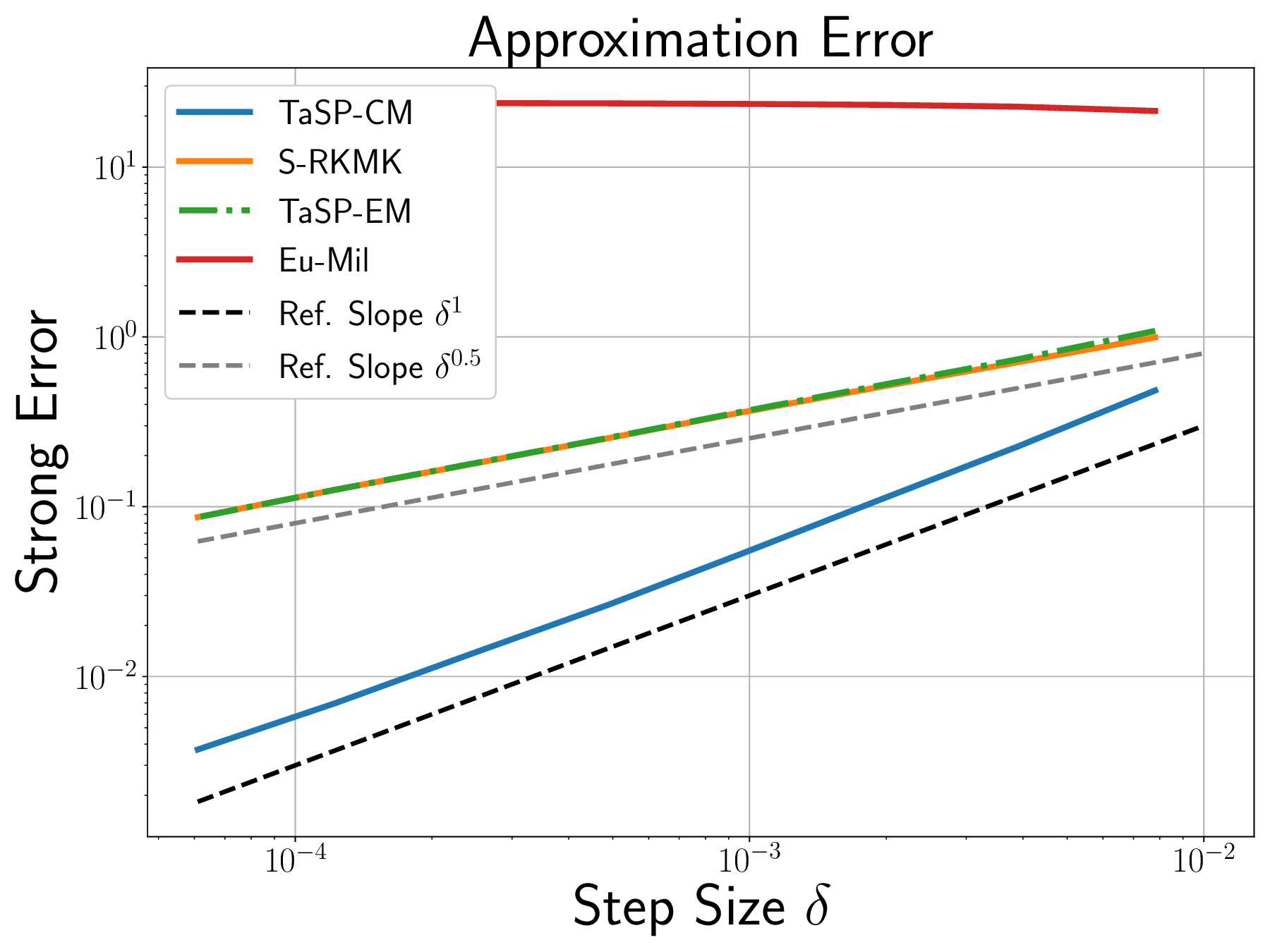}
        \caption{Mean-square error vs.\ step size $\delta$}
        \label{fig:sen-err}
    \end{subfigure}
    \hfill
    \begin{subfigure}[t]{0.45\textwidth}
        \centering
        \includegraphics[width=\linewidth]{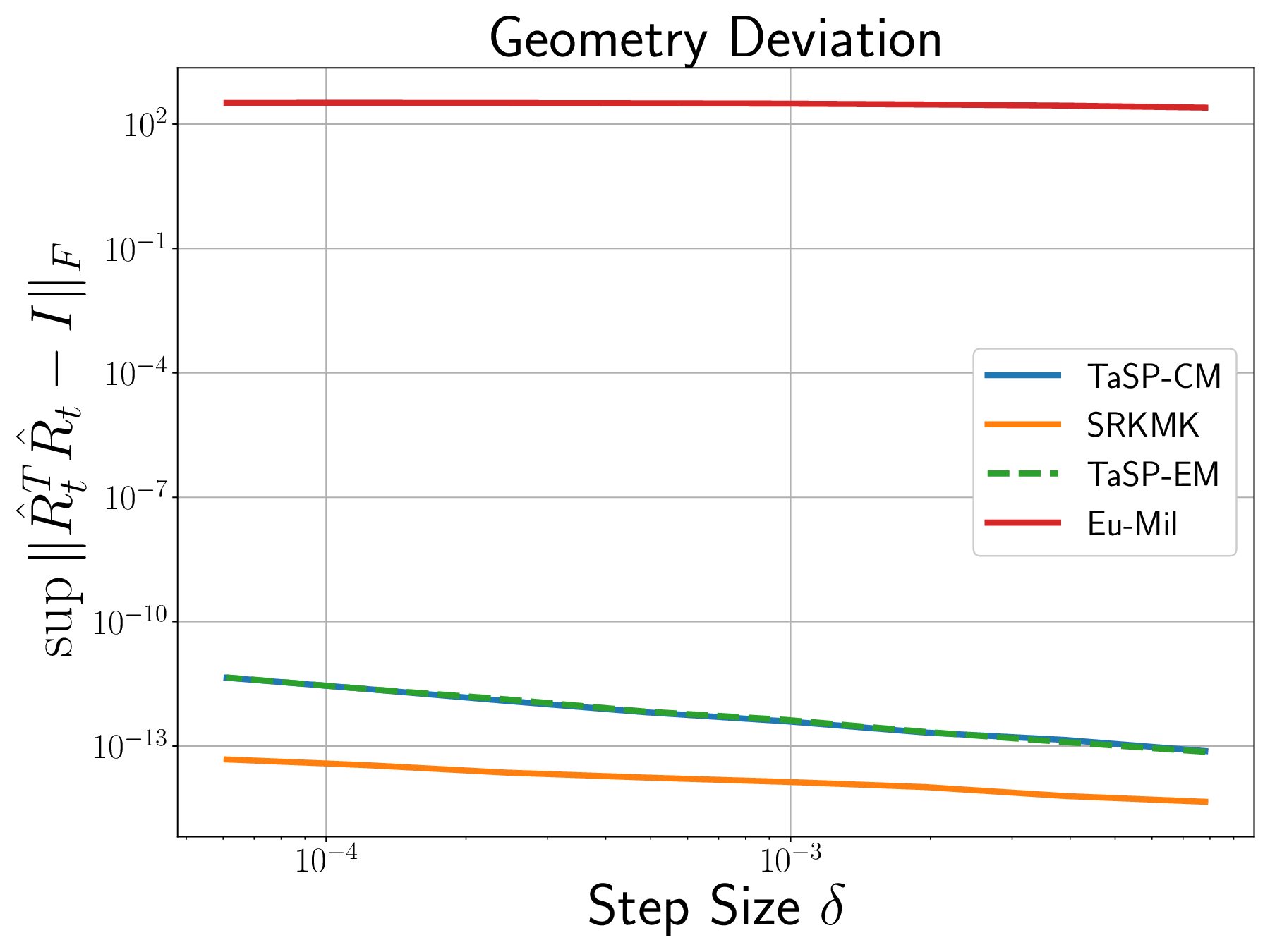}
        \caption{Geometry deviation vs.\ step size $\delta$}
        \label{fig:sen-gp}
    \end{subfigure}
    \caption{Numerical performance of TaSP–CM and competing methods over SE(n). Left: convergence rate. Right: geometric preservation.}
    \label{fig:sen-compare}
\end{figure}

\Cref{fig:bm-err} presents a log--log plot of the strong error versus the step size
$\delta$.
In the non-commutative noise setting, the proposed TaSP--CM scheme attains a
strong convergence rate of $1$, consistent with the commutative case.
By contrast, TaSP--EM achieves only order $1/2$, the S--RKMK method exhibits an
initial order $1/2$ followed by a noticeable bias, and the Euclidean Milstein
scheme fails to converge.

\Cref{fig:bm-gp} reports the geometric deviation
$\|\bfR^{\top}\bfR-I\|_F$ as a function of $\delta$.
Both TaSP--CM and TaSP--EM maintain geometric errors on the order of $10^{-11}$,
while S--RKMK achieves slightly higher accuracy at the level of $10^{-13}$; in
contrast, the Euclidean Milstein method does not preserve orthogonality.
These observations are consistent with the theoretical analysis and illustrates
the accuracy and geometry-preserving properties of the TaSP-CM scheme.

Overall, the two sets of experiments confirm that the TaSP–CM scheme achieves SCR $1$ while preserving the manifold constraint under both commutative and non-commutative noise over $\SOn$. These results illustrates our theoretical findings.




\subsection{Underwater Vehicle over SE(3)}
We consider a rigid-body underwater vehicle whose pose $\bfE(t)=(\bfR(t),\bf p(t))\in{\rm SE(3)}$ follows the It\^o SDE in~\eqref{eq:I-SDE-sen} with
\begin{align*}
    &\big(\B_0^{\bfR}(\bfE), \B_0^{\bf p}(\bfE)\big) = \Big( \Z_0+\sum_{j=1}^6 \Der{\DE}{\B_j}{\B_j}{\bfE} , v_0 \Big) \\
    &\big(\B_1^{\bfR},\B_1^{\bf p}\big)=\big(\sigma_R g(x)\tg E_{12}, 0\big),\quad \big(\B_2^{\bfR},\B_2^{\bf p}\big)=\big(\sigma_Rg(x)\tg E_{13}, 0\big) \\
    &\big(\B_3^{\bfR},\B_3^{\bf p}\big)=\big(\sigma_Rg(x)\tg E_{23}, 0\big), \quad \big(\B_4^{\bfR},\B_4^{\bf p}\big)=\big(0 ,e_1\big),\\
    &\big(\B_5^{\bfR},\B_5^{\bf p}\big)=\big(0, e_2\big),\quad \quad  \,\,\,\,\,\,\,\,\,\,\,\,\,\, \big(\B_6^{\bfR},\B_6^{\bf p} \big)=\big(0, e_3\big).
\end{align*}
The above SDE models an underwater vehicle with constant body-frame angular and linear velocities $\Z_0,v_0$ and a rotation-state–dependent diffusion. Because $\tg E_{12},\tg E_{13},\tg E_{23}$ do not commute, the rotational noise is non-commutative.

We simulate on $[0,0.5]$ with $\delta=2^{-k}$, $k=7,\dots,13$ with $50$ Monte Carlo trials, and a reference path at $\delta_{\rm ref}=2^{-18}$. We also compare our TaSP-CM method with S-RKMK, TaSP-EM and Eu-Mil schemes.

We plot the strong error $\E\big[\sup_{0\le t\le T}\|\bfE(t)-\widehat{\bfE}(t)\|_F^2\big]^{1/2}$, and the geometry deviation $\sup\|\widehat{\bfR}^\top(t)\widehat{\bfR}(t)-I\|_F$ versus $\delta$ in~\cref{fig:sen-compare}. The result shows that the TaSP–CM method exhibits SCR $1$, TaSP–EM and S-RKMK exhibits SCR $1/2$, while Eu–Mil does not converge. Geometry deviation for TaSP–CM remains at $O(10^{-11})$ across all tested steps. which also illustrate our TaSP effectiveness over $\SEn$.

\section{Conclusions}\label{sec:conclusions}
In this work, we study the numerical integration of geometric stochastic differential equations on $\mathrm{SO}(n)$ and $\mathrm{SE}(n)$ and focus on a long-standing gap in which no computationally tractable scheme simultaneously preserves the underlying geometry and attains strong convergence order greater than $1/2$. To bridge this gap, we propose the Tangent-Space-Parametrization Corrected Milstein (TaSP-CM) method, which combines a corrected Euclidean Milstein increment with a fourth-order TaSP reconstruction. We show that the TaSP-CM scheme preserves the geometry and achieves strong convergence order $1$ under both commutative and non-commutative noise, thereby extending the classical Milstein method to curved manifolds. Numerical experiments further illustrate the theoretical guarantees.

The proposed method expands the class of high-accuracy, structure-preserving integrators for simulation and inference on Lie groups, with potential applications in robotics, computer vision, molecular dynamics, and stochastic control. In future work, we will consider the extensions of our TaSP-CM scheme to other matrix manifolds, including the Stiefel and affine groups. In addition, we will attempt to adapt of the TaSP-CM framework to Riemannian diffusion models and neural SDEs, where higher-order accuracy and strict geometry preservation may enhance training stability and sample quality.

\begin{appendices}
\crefalias{section}{appendix}
\section{Technical Lemmas}\label{app:tech}

This appendix contains the technical lemmas for further proofs of the main lemmas and theorems in the article. Unless otherwise stated, throughout the Appendices~\ref{app:tech}-~\ref{app:conv-sen}.
\begin{itemize}
    \item $\|\cdot\|$ denotes the Frobenius norm;
    \item $C$ denotes a sufficiently large generic positive constant that may change from line to lines;
    \item $\{\mathcal F_t\}_{t\ge0}$ is the natural filtration generated by the Wiener processes $\{W_j\}_{j=1}^d$, i.e.
      \begin{align*}
          \mathcal F_t :=\sigma\bigl\{W_j(s) : 0\le s\le t,\ j=1,\dots,d\bigr\},
          \qquad 0 \le t \le T.
      \end{align*}
    Consequently, for the time $t$ we have
    \begin{align*}
    \E[\Dms W_j(t)|\mathcal F_{t_\ms}] =0
    \end{align*}
  and $\Dms W_j(t)$ is independent of $\mathcal F_{t_m}$;
  \item All arguments are carried out with respect to the filtration $\{\mathcal F_t\}_{t\ge0}$.
\end{itemize}

\begin{lemma}[It\^o Isometry~\cite{le2016brownian}]\label{lem:ito-isometry}
Let \(\{W(t)\}_{t\ge 0}\) be a one-dimensional standard Wiener process.  
For every square-integrable function \(f,g\in L^{2}([0,T])\),
\[
\mathbb{E}\left[\int_{0}^{T} f(t)dW_t \int_{0}^{T} g(t)dW_t\right] =  \int_{0}^{T} f(t)g(t) dt.
\]
\end{lemma}

\begin{lemma}[Doob’s maximal inequality~\cite{le2016brownian}]\label{lemma:doob}
Let $(M_{t})_{t\ge0}$ be a martingale respect to the filtration $\{\mathcal F_t\}_{t\ge0}$.  
Then, for every $t>0$,
\[
\mathbb E\left[\sup_{0\le t\le T}\|M_{t}\|^{2}\right] \le 4 \mathbb E\bigl[\|M_{T}\|^{2}\bigr].
\]
\end{lemma}

\begin{lemma}[Gr\"onwall inequality~\cite{gronwall1919note}]\label{lemma:gronwall}
Let $T>0$ and let $u:[0,T]\to[0,\infty)$ be an absolutely continuous function satisfying
\[
u(t)\le \alpha(t)+\beta\int_{0}^{t}u(s)ds,\qquad 0\le t\le T,
\]
where $\alpha:[0,T]\to[0,\infty)$ is non-decreasing and $\beta\ge 0$ is a constant.  
Then
\[
u(t)\le \alpha(t)e^{\beta t},\qquad 0\le t\le T.
\]
\end{lemma}

\begin{lemma}[Moment identities for centered Gaussians]\label{lemma:gauss-moments}
Let $\Phi\sim\mathcal N(0,\sigma)$ be a one–dimensional centered Gaussian random variable with variance $\sigma$, then following identities hold.
\begin{enumerate}[(a)]
    \item For every integer $n\ge 1$,
          \[
              \E[\Phi^{2n}] = (2n-1)!!\sigma^{n},
          \]
          where $(2n-1)!!=(2n-1)(2n-3)\cdots3\cdot1$.
    \item For every integer $n\ge 1$, there exists a constant $C_n$ depending only on $n$ such that
          \[
              \bigl\|\E[(\Phi^{2}-\sigma)^{n}]\bigr\| = \sum_k \begin{pmatrix}
                  n \\ k
              \end{pmatrix}\E[\Phi^{2k}]\sigma^{n-k}\le C_n\sigma^{n}.
          \]
\end{enumerate}
\end{lemma}

\begin{lemma}[High-order sum inequality]\label{lemma:high-ord-sum}
Let $\{u_k\}_{k=1}^d \subset \mathbb{R}^n$.  Then for any even integer $p \ge 2$, there holds that
\[
   \Big\| \sum_{k=1}^{d} u_k \Big\|^p \le d^{p-1} \sum_{k=1}^{d} \|u_k\|^p.
\]
\end{lemma}

\begin{lemma}[Normal Residual Approximation\cite{wang2025tangent}]\label{lemma:normal-error}
Let $\tg Z\in\mathfrak{so}(n)$ be such that $I-\tg Z^{\top}\tg Z$ is positive definite, and let $\nm C(\tg Z)$ denote the normal adjustment defined in Lemma \ref{lemma:correction}. If $M=\mathcal O(\delta^{-1})$, then
\[
\|\nm C(\Z)- \SSO(\Z,\Z)\| \le \half\|\Z\|^4.
\]
\end{lemma}

\begin{lemma}[Moment bound for a martingale]\label{lemma:error-sum-mar} Let $\{E_k\}_{k=1}^{M}\subset\mathbb{R}^{d}$ be random vectors and define the partial sums $ H_m := \sum_{k=1}^{m} E_k,  0 \le m \le M,$ where $M = \mathcal{O}(\delta^{-1})$ for some time–step parameter $\delta > 0$. If the process $\{H_m\}$ is a martingale w.r.t the filtration $\{\mathcal F_{t_m}\}$ and $\E[\|E_m\|^{2}] = \mathcal{O}(\delta^{3})$ for every $m\le M$, then
\[
    \E\Bigl[\sup_{0\le m\le M} \|S_m\|^{2}\Bigr] 
    = \mathcal{O}(\delta^{2}).
\]
\end{lemma}

\begin{proof}
If $\{H_m\}$ is a martingale with respect to $\{\mathcal F_{t_m}\}$ and 
$\E[\|E_m\|^{2}] = \mathcal O(\delta^{3})$,
then Doob’s maximal inequality (\cref{lemma:doob}) gives
\[
    \E\Bigl[\sup_{0\le m\le M}\|H_m\|^{2}\Bigr]
    \le
    4\E\bigl[\|H_M\|^{2}\bigr].
\]
Martingale orthogonality implies
\[
    \E\bigl[\|H_M\|^{2}\bigr]
    =
    \sum_{k=1}^{M}\E[\|E_k\|^{2}]
    =
    M\mathcal O(\delta^{3}),
\]
because $M=\mathcal O(\delta^{-1})$.  
Hence
$\E[\sup_{m}\|H_m\|^{2}] = \mathcal O(\delta^{2})$.
\end{proof}

\begin{lemma}[Moment bound for partial‐sum]\label{lemma:error-sum-part} Let $\{E_k\}_{k=1}^{M}\subset\mathbb{R}^{d}$ be random vectors where $M = \mathcal{O}(\delta^{-1})$ for some time–step parameter $\delta > 0$. If $\E[\|E_m\|^{2}] = \mathcal{O}(\delta^{4})$ for every $m\le M$, then
\[
    \E\Bigl[\sup_{0\le m\le M} \|S_m\|^{2}\Bigr] 
    = \mathcal{O}(\delta^{2}).
\]
\end{lemma}

\begin{proof}
Assume that $\E[\|E_k\|^{2}] = \mathcal O(\delta^{4})$ for all $k$.
By Cauchy–Schwarz,
\[
    \|S_m\|^{2}
    =
    \Bigl\|\sum_{k=1}^{m}E_k\Bigr\|^{2}
    \le
    m\sum_{k=1}^{m}\|E_k\|^{2}
    \le
    M\sum_{k=1}^{M}\|E_k\|^{2},
\]
and therefore
\[
    \E\Bigl[\sup_{0\le m\le M}\|S_m\|^{2}\Bigr]
    \le
    M\sum_{k=1}^{M}\E[\|E_k\|^{2}]
    =
    M^{2}\mathcal O(\delta^{4})
    =
    \mathcal O(\delta^{2}).
\]
In either case we obtain  
$\E[\sup_{0\le m\le M}\|S_m\|^{2}] = \mathcal O(\delta^{2})$,  
which completes the proof.
\end{proof}

\section{Proof of Lemma~\ref{lemma:1}}\label{app:lemma1}
\begin{proof}
We put the whole proof into four steps.\\

\noindent\textbf{Step 1: Local expansion of the TaSP-CM series.}

We represent TaSP-CM solution $\hatR(t)$ as an accumulated second-order Taylor expansion along the TaSP–CM iterates $\{\hRms\}$. For any $t\le T$, denote $m'=\sumub$, \cref{alg:tasp-ms} allows us to write the second-order Taylor expansion of TaSP-CM solution $\hatR(t)$ as
\begin{align*}
    \hatR(t) = & \hRmps\sqrt{I-\hZmpos^{\top}\hZmpos} = \hRmps + \hRmps \hZmpos + \half \hRmps \hZmpos^2 + \ermn{0} \\
              = & \hatR_0 + \sum_{m=0}^{m'} \hRms \hZmos + \half \sum_{m=0}^{m'} \hRms \hZmos^2 + \sum_{m=0}^{m'}  \ermn{0},
\end{align*} where $\ermn{\nm N} =\hRms\Big(\sqrt{I-\hZmpos^{\top}\hZmpos} -I - \hZmos - \half  \hZmos^2\Big)$ is the local remainder in the normal direction. 

Recalling the definition of $\hZms$, we have
\begin{align*}
     &\sum_{m=0}^{\sumub}  \hRms \hZmos =  \sum_{m=0}^{\sumub} \hRms \hBIms \Dms t + \sum_{m=0}^{\sumub} \hRms \sum_{j=1}^d\hBjms (\Dms W_j) \\
                 \nonumber & + \half \sum_{m=0}^{\sumub} \hRms \sum_{j \neq j'} \hDjjpms (\Dms W_j \cdot \Dms W_{j'})+ \half \sum_{m=0}^{\sumub} \hRms \sum_{j=1}^d\hDjjms \left( (\Dms W_j)^2 - \Dms t \right).
\end{align*}
Under the commutative noise we have $\hDjjpms = \hDjpjms$ and $\Dms W_j \cdot \Dms W_{j'} = \Ijjpms + \Ijpjms$ for $j\neq j$. Grouping terms gives
\begin{align}\label{eq:z1}
    \nonumber \sum_{m=0}^{\sumub} & \hRms \hZmos  = \sum_{m=0}^{\sumub} \hRms \hBIms \Dms t + \sum_{m=0}^{\sumub} \hRms \sum_{j=1}^d\hBjms  (\Dms W_j)  \\ 
    & +\sum_{m=0}^{\sumub} \hRms \sum_{j \neq j'} \hDjjpms \Ijjpms+ \half \sum_{m=0}^{\sumub} \hRms \sum_{j=1}^d\hDjjms \left( (\Dms W_j)^2 - \Dms t \right). \nonumber \\ 
    & = \sum_{m=0}^{\sumub}  \hRms \hBIms \Dms t + \sum_{m=0}^{\sumub} \hRms \sum_{j=1}^d\hBjms  (\Dms W_j) + \sum_{m=0}^{\sumub} \hRms \sum_{j,j'=1}^d \hDjjpms \Ijjpms.
\end{align}
Similarly, with the SFF $\SS_{j,j'}^m = \SSO(\hBjms,\hBjpms) = \half (\hBjms\hBjpms +\hBjpms\hBjms)$, we obtain
\begin{align}\label{eq:z2} 
     \sum \hRms \hZmos^2 =  & \sum_{m=0}^{\sumub} \hRms \sum_{j\neq j'} \hBjms\hBjpms  (\Dms W_j \cdot \Dms W_{j'})    \nonumber  \\ 
      & + \sum_{m=0}^{\sumub} \hRms \sum_{j=1}^d(\hBjms)^2(\Dms W_j)^2  + \sum_{m=0}^{\sumub} \varepsilon_m \nonumber \\
    = &\sum_{m=0}^{\sumub} \hRms \sum_{j,j'=1}^d \SS_{j,j'}^m \Ijjpms + \sum_{m=0}^{\sumub} \hRms \sum_{j=1}^d\SS_{j,j}^m(\Dms t)^2  + \sum_{m=0}^{\sumub} \varepsilon_m,   
\end{align}
where $ \varepsilon_m =\hRms \Big(\hZmos^2 -  \sum_{j,j'=1}^d \SS_{j,j'}^m \Ijjpms - \sum_{j=1}^d\SS_{j,j}^m(\Dms t)^2\Big) $ is the higher-order remainder.

Combining \cref{eq:z1} and \cref{eq:z2} yields
\begin{align*}
    & \hatR(t) =  \sum_{m=0}^{\sumub} \hRms \hBIms \Dms t + \sum_{m=0}^{\sumub} \hRms \sum_{j=1}^d\hBjms (\Dms W_j) \\
                & + \sum_{m=0}^{\sumub} \hRms \sum_{j,j'=1}^d \left(\hDjjpms + \SS_{j,j'}^m\right) \Ijjpms + \sum_{m=0}^{\sumub} \hRms \sum_{j=1}^d (\hBjms)^2 \Dms t + \sum_{m=0}^{\sumub} \varepsilon_m + \sum_{m=0}^{\sumub} \ermn{\nm N}.
\end{align*}
\\

\noindent\textbf{Step 2: Comparing $\hatR(t)$ with the Frozen Milstein series $\bfR_\delta(t)$.}

For any $ 0 < t < T$, we rewrite the frozen-field Milstein series $\bfR_\delta(t)$ as
\begin{align*}
    \bfR_\delta(t) = & \hatR_0 + \sum_{m=0}^{m'} \hRms \hBIms \Dms t +  \half \sum_{m=0}^{m'}  \sum_{j=1}^d \hRms (\hBjms)^2 \Dms t \\
    & + \sum_{m=0}^{m'}  \sum_{j=1}^d \hBjms \Dms W_j(t) +  \sum_{m=0}^{m'}  \sum_{j,j'=1}^d \hnjjpms \Ijjpms.
\end{align*}
From the tangent and normal decomposition, we know that
\begin{align*}
     \hnjjpms = \POtg(\hnjjpms)+\POnm(\hnjjpms)  =\hDjjpms + \SS_{j,j'}^m,
\end{align*}
which implies
\begin{align*}
    \bfR_\delta(t) +& \sum_{m=0}^{\sumub} \hRms \hBIms \Dms t + \sum_{m=0}^{\sumub} \hRms \sum_{j=1}^d\hBjms (\Dms W_j) \\
    & +\sum_{m=0}^{\sumub} \hRms \sum_{j \neq j'} \hnjjpms \Ijjpms + \sum_{m=0}^{\sumub} \hRms \sum_{j=1}^d (\hBjms)^2 \Dms t 
\end{align*}
and hence we obtain
\begin{align}\label{eq:app-B-combine-1}
     \hatR(t)  =  \bfR_\delta(t) +  \sum_{m=0}^{\sumub} \varepsilon_m + \sum_{m=0}^{\sumub} \ermn{\nm N}.
\end{align}
Hence it is suffcient to show
\begin{align*}
    \EsupN{\varepsilon_m } = \EsupN{\ermn{\nm N}} = \calO(\delta^2)
\end{align*}
\noindent\textbf{Step 3: Estimating  $\ermn{\nm N}$.}

Now we analyze the residuals
\begin{align*}
    \mathbb{E}\Bigl[ \sup_{0 \le t\le T} \big\| \sum_{m=0}^{\sumub} \ermn{\nm N} \big\|^2 \Bigr] , 
\end{align*}
and prove it to be $\calO(\delta^2)$. Because $\hatR_{m}\in\SOn$, we have $ \|\hatR_{m}A\|=\|A\|$ for any matrix $A$. Therefore, we have 
\begin{align*}
   \| \ermn{\nm N} \| = & \Big\|\hRms\Big(\sqrt{I+\hZmos^2} -I - \hZmos - \half  \hZmos^2\Big)\Big\| \\
                      = & \Big\|\Big(\sqrt{I+\hZmos^2} -I - \hZmos - \half  \hZmos^2\Big)\Big\| 
\end{align*}
\Cref{lemma:normal-error} immediately gives 
\begin{align*}
  & \|\ermn{\nm N}\|^{2} = \|\nm C(\hZmos)- \SSO(\hZmos,\hZmos)\|^2 \le \|\hZmos\|^{8}  \\[3pt]
    \le & \Big\| \B_I \Dms t + \sum_{j=1}^d \B_j \Dms W_{j} + \half \sum_{j, j'} \hDjjpms ( \Dms W_j \Dms W_{j'} - \delta_{j,j'} \Dms t ) \Big\|^8 
\end{align*}
Applying~\cref{lemma:high-ord-sum}, we have
\begin{align*}
    \|\ermn{\nm N}\|^{2} \le &  (d^2+d+1)^7\Big( \| \B_I \Dms t\|^8 +  \sum_{j=1}^d \|\B_j \Dms W_{j}\|^8\Big) \\
    & + (d^2+d+1)^7 \sum_{j, j'} \hDjjpms  \Big\| \Dms W_j\Dms W_{j'} - \delta_{j,j'} \Dms t  \Big\|^8 
\end{align*}
From Assumptions~\ref{asm:drift} and~\ref{asm:diffusion}, there exists a constant $C$ (as mentioned, $C$  a sufficiently large constant that varies from line to line) such that $\| \B_I\|+\sum_{j=1}^d\|\B_j\|+ \sum_{j, j'} \hDjjpms \le C$. Consequently, we have
\begin{align*}
    \E\big[\|\ermn{\nm N}\|^{2}\big] \le   C \Big( \E\big[(\Dms t)^8\big] +  \sum_{j=1}^d  \E\big[(\Dms W_{j}(t))^8\big] +  \sum_{j, j'}  \E\big[( \Dms W_j\Dms W_{j'} - \delta_{j,j'} \Dms t )^8 \big]\Big)
\end{align*}
From~\cref{lemma:gauss-moments}, we have 
\begin{align*}
    \begin{cases}
        \mathbb E[(\Dms W_{j})^{8}]=\calO(\Dms t^4) \quad \forall j \in [d], \\
        \mathbb E[((\Dms W_{j})^{2}-\Dms t)^{8}] =\calO(\Dms t^8)  \quad \forall j \in [d], \\
        \mathbb E[(\Dms W_{j})^{8}(\Dms W_{j'})^{8}]=\calO(\Dms t^8) \quad \forall j \neq j' \in [d],
    \end{cases}
\end{align*} which implies
\begin{align*}
    \E[\|\ermn{\nm N}\|^2] \le C \Dms t^4= \calO(\delta^4).
\end{align*}
Applying \cref{lemma:error-sum-part} yields
\begin{align}\label{eq:app-B-combine-2}
   \EsupN{\ermn{\nm N}} =  \mathcal{O}(\delta^2).
\end{align}    

\noindent\textbf{Step 4: Estimating $\varepsilon_m$.}

The remainder $\varepsilon_m$ can be written as 
\begin{align*}
    \varepsilon_m = \sum_{i = 1}^{9}\ermn{i},
\end{align*}
where
\begin{align}\label{eq:vare-expansion}
    \begin{cases}
        \ermn{1} = \hatR_m\sum_{j=1}^d \SSO(\hBIms, \hBjms) (\Dms t \cdot \Dms W_j), \\
        \ermn{2} = \hatR_m\sum_{j=1}^d\sum_{l\neq l'}  \SSO(\hBjms ,\hDllpms) (\Dms W_j \cdot \Dms W_l \cdot \Dms W_l'), \\
        \ermn{3} = \hatR_m\sum_{j=1}^d\sum_{l=1}^d \SSO (\hBjms, \hDllms)(\Dms W_j) ( (\Dms W_l)^2 - \Dms t)),\\
        \ermn{4} = \hatR_m \SSO(\hBIms,\hBIms) (\Dms t)^2,  \\
        \ermn{5} = \Big( \hatR_m \sum_{j\neq j'}  \hDjjpms  (\Dms W_j \cdot \Dms W_j')\Big)^2,  \\
        \ermn{6} = \Big( \hatR_m \sum_{j=1}^d\hDjjms \big( (\Dms W_j)^2 - \Dms t \big) \Big)^2, \\
        \ermn{7} = \hatR_m\sum_{j\neq j'}  \SSO(\hBIms, \hDjjpms) (\Dms t \cdot \Dms W_j \cdot \Dms W_j'), \\
        \ermn{8} = \hRms  \sum_{j} \SSO(\hBIms, \hDjjms) \Big(\Dms t \cdot \big((\Dms W_j)^2 - \Dms t\big)\Big), \\
        \ermn{9} = \hatR_m \sum_{j\neq j'} \sum_{l} \SSO( \hDjjpms, \hDllms) (\Dms W_j \cdot \Dms W_j')((\Dms W_l)^2 - \Dms t)).
    \end{cases}
\end{align}
In the following, we will show each remainder series  $ \sum_{m=0}^{\sumub} \ermn{i}  $ contributes at most $\mathcal O(\delta^{2})$ in mean square, and hence
\begin{align*}
    \EsupN{ \varepsilon_m} \le 9 \sum_{i=1}^9  \EsupN{\ermn{i}} = \mathcal O(\delta^{2}).
\end{align*}
To analyze, we split the above terms into two groups. The first group contains $\ermn{1}$-$\ermn{3}$, while the second group contains $\varepsilon_{m}^{(4)}$–$\varepsilon_{m}^{(9)}$.

   For the group $\ermn{1}$-$\ermn{3}$, we first estimate $\ermn{1}$. Let $\nm S^m_{I,j} = \SSO(\hBIms, \hBjms)$, we define
\begin{align*}
    H_t^{(1)} = \sum_{m=0}^{\sumub}\ermn{1}= \sum_{m=0}^{\sumub}\hatR_m\sum_{j=1}^d \nm S^m_{I,j} (\Dms t \cdot \Dms W_j).
\end{align*}
Because each Wiener increment $\Dms W_{j}(t)$ is centered and independent of $\mathcal F_{t_{m}}$, we have
$\mathbb E[\ermn{1}\mid\mathcal F_{t_m}]=0$, hence $H_{t}^{1}$ is an $\{\mathcal F_{t}\}$-martingale.
Then we estimate the second moment of $\ermn{1}$. Since $\hatR_{m}\in\SOn$, we have
\begin{align*}
     \|\ermn{1}\|^2 = \|\sum_{j=1}^d \nm S^m_{I,j} (\Dms t \cdot \Dms W_j)\|^2 \le d \sum_{j=1}^d \|\nm S^m_{I,j} (\Dms t \cdot \Dms W_j)\|^2.
\end{align*}
Under Assumptions~\ref{asm:drift} and~\ref{asm:diffusion}, the second fundamental form is uniformly bounded. Hence
\begin{align*}
    \E [\|\ermn{1}\|^2] \le C \sum_{j=1}^d (\Dms t)^2 \E[(\Dms W_j)^2] = \calO(\delta^3), \quad \forall m\in[M],
\end{align*}
where the last equality holds due to~\cref{lemma:gauss-moments}. Applying~\cref{lemma:error-sum-mar}, we have
\begin{align*}
     \EsupN{\ermn{1}} = \calO(\delta^2).
\end{align*}
The argument of proving 
\begin{align*}
     \EsupN{\ermn{i}} = \calO(\delta^2), \quad  i=2,3.
\end{align*}
is similar to those in $\ermn{1}$. Here we omitted the proof for space limit.

Next, we estimate $\ermn{4}-\ermn{9}$. We begin with $\ermn{5}$. Since $\hatR_{m}\in\SOn$, we have
    \begin{align*}
        \| \ermn{5}\|^2 = &  \Big\| \Big( \hatR_m \sum_{j\neq j'} \hDjjpms (\Dms W_j \cdot \Dms W_j')\Big)^2 \Big\|^2 \\
                        \le &  \Big\|  \hatR_m \sum_{j\neq j'} \hDjjpms (\Dms W_j \cdot \Dms W_j') \Big\|^4 =  \Big\|  \sum_{j\neq j'} \hDjjpms (\Dms W_j \cdot \Dms W_j') \Big\|^4
    \end{align*}    
    By~\cref{lemma:high-ord-sum}, we have
    \begin{align*}
         \| \ermn{5}\|^2 \le d^6 \sum_{j\neq j'} \|\hDjjpms (\Dms W_j \cdot \Dms W_j') \|^4 
    \end{align*}
    Under the uniform boundedness Assumptions~\ref{asm:drift} and~\ref{asm:diffusion}, we have
    \begin{align*}
        \E\big[\| \ermn{5}\|^2\big] \le C\sum_{j\neq j'} \E\big[(\Dms W_j)^4 (\Dms W_j')^4\big].
    \end{align*}
    Independence of Wiener increments gives
    \begin{align*}
       \mathbb E \bigl[(\Delta W_{j}^{m}\Delta W_{j'}^{m})^{4}\bigr] = \mathbb E[(\Delta W_{j}^{m})^{4}] \mathbb E[(\Delta W_{j'}^{m})^{4}] =3(\Dms t)^4 = \calO(\delta^4) \quad j\neq j'.
    \end{align*}
    Therefore, from~\cref{lemma:error-sum-part}, we obtain
    \begin{align*}
        \EsupN{\ermn{5}} = \calO(\delta^2).
    \end{align*}
The argument of proving 
\begin{align*}
     \EsupN{\ermn{i}} = \calO(\delta^2), \quad  i=4,6,7,8,9
\end{align*}
is similar to those in $\ermn{5}$. Here we omitted the proof for space limit.

Finally, we conclude that
\begin{align}\label{eq:app-B-combine-3}
    \EsupN{ \varepsilon_m} \le 9 \sum_{i=1}^9  \EsupN{\ermn{i}} = \mathcal O(\delta^{2}).
\end{align}
Combining~\eqref{eq:app-B-combine-1}--~\eqref{eq:app-B-combine-3}, we have
\begin{align*}
    \E\big[ \sup_{0 \le t\le T}\| \hatR(t)- \bfR_\delta(t) \|^2 \big] \le 2 \EsupN{(\ermn{\nm N} + \varepsilon_m)} = \calO(\delta^2),
\end{align*}
which completes our proof.
    
\end{proof}

\section{Proof of Lemma~\ref{lemma:2}}\label{app:lemma2}
\begin{proof} The whole proof is divided in three steps. \medskip

\noindent\textbf{Step 1:  It\^o-Taylor expansion of the exact solution $\bfR(t)$.}

We define the path-wise errors by
\begin{align*}
    \begin{cases}
        \alpha(t) = \|\hatR(t) - \hatR_\delta(t)\|, &\quad \bar \alpha(t) = \suptpt \alpha(t) \\
        \zeta(t)  = \|\hatR(t) - \hatR_\delta(t)\|, &\quad \bar \zeta(t) = \suptpt \gamma(t).
    \end{cases}
\end{align*}

Let $\njjpms = \Der{\nabla}{\bfR \B_j}{\bfR \B_{j'}}{\bfR},  j,j'\in [d]$. On the time interval $t\in[m\delta,(m+1)\delta]$, the exact solution $\bfR(t)$ admits the It\^o-Taylor expansion
\begin{align*}
     \bfR(t) = & \bfR_m + \Rms \BIms (
     \Dms t) + \half \sum_{j=1}^d \Rms (\Bjms)^2 \Dms t  \\
     & + \sum_{j=1}^d\Rms \Bjms \Dms W_j(t) + \sum_{m=0}^{\sumub} \Rms \sum_{j,j'=1}^d \njjpms \Ijjpms + \eprm 
\end{align*}
where $\eprm$ collects all higher-order residuals.

Compared with~\cref{eq:frozen-mil}, we obtain
\begin{align}\label{eq:to-be-sum-1}
    \bfR(t) - \bfR(m\delta) = \hatR_\delta(t) - \hatR_{\delta}(m\delta) + \sum_{i=1}^5 A^{(i)}_m + \eprm, 
\end{align}
where
\begin{align*}
\begin{cases}
    A^{(1)}_m = (\Rms \BIms - \hRms \hBIms) \Dms t\\ 
    A^{(2)}_m =  \sum_{j=1}^d(\Rms (\Bjms)^2 - \hRms (\hBjms)^2) \Dms t \\
    A^{(3)}_m =  \sum_{j=1}^d(\Rms \Bjms - \hRms \hBjms))\Dms W_j \\
    A^{(4)}_m = \sum_{j\ne j'} (\Rms \njjpms - \hRms \hnjjpms)(\Dms W_j \Dms W_{j'}) \\
    A^{(5)}_m = \sum_{j=1}^d (\Rms \njjms - \hRms \hnjjms)((\Dms W_j)^2-\Dms t).
\end{cases}
\end{align*}
Summing~\cref{eq:to-be-sum-1} over $m$, we have.
\begin{align*}
    \bfR(t) - \hatR_\delta(t) =& \bfR_0 + \sum_{m=0}^{\sumub}\sum_{i=1}^5 A^{(i)}_m+ \sum_{m=1}^{\sumub}\eprm.
\end{align*}
Thus, we obtain
\begin{align}{2}\label{eq:zeta}
 \E[\bar{\zeta}^2(t)]\le 6 \sum_{i=1}^5 \EsupN{A^{(i)}_m} +  6\EsupN{\eprm}. & 
\end{align}
\\

\noindent\textbf{Step 2:  Estimating $A^{(i)}_m$.}

We first estimate $\EsupN{A^{(1)}_m}$. From Young's inequality, we have
\begin{align*}
    \big\|\sum_{m=0}^{\sumub} A^{(1)}_m\big\|^2 \le & \sumub  (\Dms t)^2 \sum_{m=0}^{\sumub} \|\Rms \BIms - \hRms \hBIms\|^2 \\ 
    \le & T (\Dms t)\sum_{m=0}^{\sumub} \|\Rms \BIms - \hRms \hBIms\|^2.
\end{align*}
where $T$ denotes the time horizon.  
From~\cref{asm:diffusion}, $\B_I(\bfR)$ is Lipschitz continuous, and hence so is $\bfR\B_I(\bfR)$.  
Therefore, there exists a constant $C>0$ such that
\begin{align*}
     \|\Rms \BIms - \hRms \hBIms\|^2 \le C\| \bfR_m - \hatR_m\|^2 \le C \|\bar\alpha(t_m)\|^2  + C\|\bar\zeta(t_m)\|^2
\end{align*}
Furthermore, we have
\begin{align*}
     \EsupN{ A^{(1)}_m} \le  C \sum_{m=0}^{\sumub}(\E\big[\|\bar\alpha(t_m)\|^2\big]\Dms t)  + C\sum_{m=0}^{\sumub}(\E\big[\|\bar\zeta(t_m)\|^2\big]\Dms t)
\end{align*}
Since $\bar\alpha(t)$ and $\bar\zeta(t)$ are non-decreasing, we have
\begin{align*}
     \sum_{m=0}^{\sumub}(\E\big[\|\bar\alpha(t_m)\|^2\big]\Dms t) \le T \E[\bar \alpha^2(T)],\quad \sum_{m=0}^{\sumub}(\E\big[\|\bar\zeta(t_m)\|^2\big]\Dms t) \le  \int_{0}^t \E[\bar \zeta^2(t')] dt'.
\end{align*}
Therefore, we have
\begin{align}\label{eq:A1}
      \EsupN{ A^{(1)}_m} \le C \int_{0}^t \E[\bar \zeta^2(t')] dt' + C T \E[\bar \alpha^2(T)].
\end{align}
Follow the similar procedure, we can also obtain that
\begin{align}\label{eq:A2-A5}
    \EsupN{ A^{(i)}_m}\le C \int_{0}^t \E[\bar \zeta^2(t')] dt' + C T \E[\bar \alpha^2(T)],\quad  i=2,3,4,5.
\end{align}

Substituting~\cref{eq:A1,eq:A2-A5} into~\cref{eq:zeta} yields
\begin{align*}
    \E[\bar{\zeta}^2(t)] \le  C \int_{0}^t \E[\bar \zeta^2(t')] dt' + C\E[\bar \alpha^2(T)] + \E\big[ \suptpt \|\eprm\|^2 \big].
\end{align*}
Applying Gr\"onwall’s inequality (\ref{lemma:gronwall}) gives
\begin{align}\label{eq:app-C-combine-1}
    \E[\bar{\zeta^2}(t)] \le  e^{CT} \E[\bar \alpha^2(T)] + e^{CT}\E\big[  \sup_{0\le t'\le T}  \|\eprm\|^2 \big].
\end{align}
\\

\noindent\textbf{Step 3:  Estimating $\eprm$.}

From~\cref{lemma:1}, we have $\E[\bar \alpha^2(T)] = \calO(\delta^2)$. It is sufficient to show
\begin{align*}
    \EsupN{\eprm} = \calO(\delta^2)
\end{align*}
We prove the above claim in the following proof. For $j \in [d]$, we define the differential operators
\begin{align*}
   & \mathfrak L_0 f=\nabla_{\bfR B_{I}(\bfR)} f + \thalf\nabla_{\bfR \nm \Sigma(\bfR)} f+\thalf\sum_{j=1}^{d}\nabla^{2}f(\bfR)[\bfR \B_{j},\bfR \B_{j}], \quad
   & \mathfrak L_j f=\nabla_{\bfR \B_{j}} f.
\end{align*}
and integral differential operators
\begin{align*}
    & I_{\alpha,\beta}^m(f)=\int_{t_\ms}^{t_\mos}\int_{t_\ms}^{s_1} f d\chi_\beta(s_1) d\chi_\alpha(s_2), \\ & I_{\alpha,\beta,\gamma}^m(f)=\int_{t_\ms}^{t_\mos}\int_{t_\ms}^{s_1}\int_{t_\ms}^{s_2} f d\chi_\gamma(s_1) d\chi_\beta(s_2)d\chi_\alpha(s_3), 
\end{align*}
where $d\chi_0 = ds$ and $d\chi_j = dW_j$. Accoriding to~\cite{kloeden1992numerical}, the residual of the It\^o-Taylor expansion $\eprm$ is written as
\begin{align*}
    &\eprm =  \eprmn{1}+\eprmn{2}, \quad \eprmn{1} = I_{0,0}\big(\mathfrak L_0\mathfrak L_0(\bfR)\big)\\
    &\eprmn{2}=  \sum_{j=1}^d I_{j,0} \big(\mathfrak L_j\mathfrak L_0(\bfR)\big) +   \sum_{j=1}^d I_{0,j} \big(\mathfrak L_0\mathfrak L_j(\bfR)\big) + \sum_{j=0}^d\sum_{j',j''=1}^d I_{j,j',j''}\big(\mathfrak L_j\mathfrak L_{j'}\mathfrak L_{j''}(\bfR)\big).
\end{align*}
Under~\cref{asm:drift,asm:diffusion} and the standard moment estimates for multiple It\^o integrals~\cite{kloeden1992numerical}, we obtain
\begin{align*}
    \mathbb E\bigl[\|\eprmn{1}\|\bigr]=\mathcal O(\delta^{4}),
\end{align*}
and the partial sum $\sum_{m}\eprmn{2}$ is a martingale with
\begin{align*}
    \mathbb E\bigl[\|\eprmn{2}\|\bigr]=\mathcal O(\delta^{3}).
\end{align*}
Applying~\cref{lemma:error-sum-mar,lemma:error-sum-part} gives
\begin{align}\label{eq:app-C-combine-2}
      \mathbb E\Bigl[
     \sup_{0\le t\le T}
     \Bigl\|\sum_{m=0}^{\sumub}\varepsilon_m'\Bigr\|^{2}
  \Bigr]
  =\mathcal O(\delta^{2}).
\end{align}

Substituting~\cref{eq:app-C-combine-2} into~\cref{eq:app-C-combine-1}, we obtain
\begin{align*}
     \E[\bar{\zeta^2}(t)]  = \calO(\delta^2),
\end{align*}
which completes the proof. 
\end{proof}

\section{Proof of Theorem~\ref{thm:convergence-non-comm}}\label{app:non-commute}
Before carrying out the proof of Theorem~\ref{thm:convergence-non-comm}, we state a number of lemmas that describe the moments of the double It\^o integral $\Ijjpms(t)$ and approximation error $\DIjjpms$.
\subsection{Moment Estimate of Double It\^o integral}
\begin{lemma}[\cite{kloeden1992numerical}]\label{lem:I2N} For $0 \le t \le T$ and all $j \neq j'\in[d]$
\begin{align*}
    \E[\Ijjpms(t)] = 0,
    \qquad
    \E\bigl[(\Ijjpms(t))^{2n}\bigr]
    =
    \bigl(2(2n-1)e^{T}\bigr)^{2n}(\Delta_\ms t)^{2n}.
\end{align*}
\end{lemma}

\begin{lemma}\label{lem:I1W}
     For every index $j \neq j\in [d]$, and $l\in [d]$, we have
     \begin{align*}
         \E[\Ijjpms(t)\Dms W_l(t)]=0,  \qquad 0\le t\le T.
     \end{align*}
\end{lemma}
\begin{proof}
If $l\notin\{j,j'\}$ the claim follows immediately from independence. So, it remains to treat $l=j$ or $l=j'$.  Recall the identity
\[
\Ijjpms + \Ijpjms = \tfrac{1}{2} \Dms W_j \cdot \Dms W_{j'}.
\]
Taking expectations gives
\[
\E[\Dms W_j \cdot \Dms W_{j'} \cdot \Dms W_{l'}] = 0.
\]
so $\mathbb{E}[\Ijpjms\Delta_\ms W_{j'}]=0$ implies $\mathbb{E}[\Ijpjms\Delta_\ms W_{j'}]=0$. W.l.o.g, we  assume $l=j'$.

When $ l = j' $, we define the process $ \Phjms(s) := \int_{t_\m}^{s} dW_j = dW_j(s) - dW_j(t_\ms)$. Conditioning on $ W_j $, we get
\begin{align*}
    \E[\Dms W_{j'} \cdot \Ijjpms \mid W_j] = \E\left[ \int_{t_\ms}^{t_\mos} dW_{j'}(s_2) \int_{t_\ms}^{t_\mos} \Phjms(s_2) dW_{j'}(s_2) \right].
\end{align*}
Applying the It\^o isometry (\cref{lem:ito-isometry}) yields
\begin{align*}
    \E[\Dms W_{j'} \cdot \Ijjpms \mid W_j]  = \int_{t_\ms}^{t_\mos} \Phjms(s)  ds.
\end{align*}
Finally taking the total expectation yields
\begin{align*} 
    \E[\Dms W_{j'} \cdot \Ijjpms] = \int_{t_\ms}^{t_\mos} \E[\Phjms(s)]  ds = 0,
\end{align*}
which completes the proof.
\end{proof}

\begin{lemma}For every index $j \neq j\in [d]$, $l\in [d]$, and time $0\le t \le T$, we have\label{lem:I2W2-I2W4}
     \begin{enumerate}[(a)]
         \item $\E[\Ijjpms^2(t)\Dms W_l^2(t)]=\calO(\delta^3)$;
         \item  $\E[\Ijjpms^2(t)\Dms W_l^4(t)]=\calO(\delta^4)$.
     \end{enumerate}
\end{lemma}
\begin{proof} Adopting the notation $\Phjms(s)$ from~\cref{lem:I1W}, we also assume $l=j$, w.l.o.g.. \\
\textbf{Case (a):} Conditioning on $W_j$, we have
\begin{align*}
    \E[\Ijjpms^2\Dms W_j^2|W_j] = \E\Bigl[ \Bigl( \int_{t_\ms}^{t_\mos} (\Phjms(s_2))(\Dms W_j) dW_{j'}(s_2) \Bigr)^2 \Big].
\end{align*}
Using~\cref{lem:ito-isometry} again, we have
\begin{align*}
     \E[\Ijjpms^2\Dms W_j^2|W_j] = &  \int_{t_\ms}^{t_\mos} (\Phjms^2(t))(\Dms W_j)^2 ds.
\end{align*}
Taking the total expectation yields
\begin{align*}
     \E[\Ijjpms^2\Dms W_j^2] = & \int_{t_\ms}^{t_\mos} \E\big[(\Phjms^2(t))(\Dms W_j)^2\big] ds.
\end{align*}
Since 
\begin{align*}
    \E\big[(\Phjms^2(s))(\Dms W_j(s))^2\big] = & \E\big[(\Phjms^2(s))(\Phjms(s) + W_j(t) - W_j(s)  )^2\big] \\
    \le & 3(s-t_{\ms})^2 + (s-t_{\ms}) (t-s) \le 4 \delta^2,
\end{align*}
we have
\begin{align*}
   \E[\Ijjpms^2\Dms W_j^2]  \le \int_{t_\ms}^{t_\mos} 4\delta^2 ds \le 4\delta^3.
\end{align*}
\\
\textbf{Case (b):} Similarly, conditioning on $W_j$, we have
\begin{align*}
    \E[\Ijjpms^2\Dms W_j^4|W_j] = & \E\Bigl[ \Bigl( \int_{t_\ms}^{t_\mos} (\Phjms(s_2))(\Dms W_j)^2 dW_{j'}(s_2) \Bigr)^2 \Big],
\end{align*}
Taking the total expectation yields
\begin{align*}
     \E[\Ijjpms^2\Dms W_j^4]  = & \int_{t_\ms}^{t_\mos} \E\big[(\Phjms(s))^2(\Dms W_j(s))^4\big] ds.
\end{align*}
Since 
\begin{align*}
    \E\big[(\Phjms^2(s))(\Dms W_j(s))^4\big] \le & 120(s-t_{\ms})^2 + 24(s-t_{\ms}) (t-s)^2 \le 144 \delta^3,
\end{align*}
we have
\begin{align*}
   \E[\Ijjpms^2\Dms W_j^4]  \le \int_{t_\ms}^{t_\mos} 144\delta^3 ds \le 144\delta^4.
\end{align*}
Combining cases (a) and (b), we complete our proof.
\end{proof}
\subsection{Moment Estimate of Approximation Error}
\begin{lemma}\label{lem:DI4}
    For any $h>1$, $\E\bigl[(\DIjjpms)^4\bigr] =\mathcal{O}(\delta^{4})$.
\end{lemma}
\begin{proof}
From~\cref{lem:DI1-DI2}, we know that
\[
    \E[(\DIjjpms)^4] = \frac{1}{2\pi} (\Dms t)^4 \E\Bigl[\sum_{r=h+1}^{\infty}\frac{1}{r}\Big(a_{j,r}b_{j',r}-a_{j,r}b_{j',r}\Big)\Bigr].
\]
Denote $X_r = \frac{1}{r}\Big(a_{j,r}b_{j',r}-a_{j,r}b_{j',r}\Big)$. Since $\{\ajr,\bjr\}$ are independent normal Gaussian variables, we have
\begin{align*}
    &\E[X_r]= 0,\qquad\E[\Xr^2]  = \frac{1}{r^2} \big(\E[\ajr^2] \E[\bjpr^2] +\E[\ajpr^2] \E[\bjr^2] \bigr) = \frac{2}{r^2}, \\
    &\E[X_r^4]  = \frac{1}{r^4} \left( \E[\ajr^4 \bjpr^4] + \E[\ajpr^4 \bjr^4] + 6 \E[\ajr^2 \ajpr^2 \bjr^2 \bjpr^2] \right) = \frac{24}{r^4}.
\end{align*}

We define $S_h := \sum_{r > h}  X_r$. Since $X_r$ are independent, it satisfies
\begin{align*}
    \E[S_h^4] 
    &= \sum_{r = h+1}^\infty \E[(X_r)^4] + 6 \sum_{r < s} \E[(X_r)^2] \cdot \E[(X_s)^2] \\
    &=  \big( \sum_{r = h+1}^\infty \frac{24}{r^4} + 6 \sum_{r < s} \frac{2}{r^2} \cdot \frac{2}{s^2} \big) \le 24 \big( \sum_{r = h+1}^\infty \frac{1}{r^4} + \sum_{r=h+1}^\infty \big( \frac{1}{r^2} \big)^2 \big) = \mathcal{O}(1).
\end{align*}
Therefore, we obtain
\begin{align*}
    \E[(\DIjjpms)^4] = \frac{1}{2\pi} (\Dms t)^4 \E[S_h^4] = \mathcal{O}(\delta^4).
\end{align*}
\end{proof}

\begin{lemma}\label{lem:DI1W-DI2W2-DI2W4}Let $h>1$. Then for every index $j\neq j' \in [d]$, and $l\in [d]$, we have
\begin{enumerate}[(a)]
    \item $ \E[\DIjjpms\Dms W_l] = 0$,
    \item $ \E[(\DIjjpms\Dms W_l)^2] = \calO(\delta^3)$,
    \item $ \E[(\DIjjpms)^2(\Dms W_l)^4] = \calO(\delta^4)$.
\end{enumerate}
\end{lemma}
\begin{proof}
\textbf{Case (a):}Because every $a_{j,r}$ and $b_{j,r}$ is independent of $\Delta_\ms W_l$, the random variables $\DIjjpms$ and $\Delta_\ms W_l$ are independent, therefore
$$\E[\DIjjpms\Dms W_l]  = \E[\DIjjpms]\E[\Dms W_l]=0. $$
\textbf{Case (b):}From the independence, we have
    \begin{align*}
        \E[(\DIjjpms\Dms W_l)^2] = \E[(\DIjjpms)^2]\E[(\Dms W_l)^2].
    \end{align*}
    \Cref{lem:DI1-DI2} gives  $\E[(\DIjjpms)^2] = \calO(\delta^2)$ for $h>1$, which implies
    \begin{align*}
        \E[(\DIjjpms\Dms W_l)^2] = \calO(\delta^2)\Dms t= \calO(\delta^3).
    \end{align*}
\textbf{Case (c):}  Using the same independence, we have
    \begin{align*}
        \E[(\DIjjpms)^2(\Dms W_l)^4] = \E[(\DIjjpms)^2]\E[(\Dms W_l)^4] = \calO(\delta^2)(\Dms t)^2 = \calO(\delta^4).
    \end{align*}
Combining cases (a)-(c), we complete the proof.
\end{proof}

Then we begin to prove Theorem~\ref{thm:convergence-non-comm}.
\begin{proof}
We put the whole proof into four steps.\\

\noindent\textbf{Step 1: Local expansion of the TaSP-CM series.}

Using the frozen-field Milstein series $\bfR_\delta(t)$ as defined in~\cref{eq:frozen-mil}, we write the error as
\begin{align*}
      \E\big[ \sup_{0 \le t\le T} \| \bfR(t) - \hatR(t) \|^2_F \big] \le 2\E\big[ \sup_{0 \le t\le T}\| \hatR(t)- \bfR_\delta(t) \|_F^2 \big] + 2\E\big[ \sup_{0 \le t\le T} \|  \bfR_\delta(t) - \bfR(t)\|^2_F \big].
\end{align*}

Similar to the proof of ~\cref{lemma:1}, we write $\hatR(t)$ as
\begin{align*}
     \hatR(t) =  \bfR_\delta(t) +  \sum_{m=0}^{\sumub}\varepsilon^{\Delta}_{m} +  \sum_{m=0}^{\sumub} \epprmn{\nm N}+ \sum_{m=0}^{\sumub} \epprm.
\end{align*}
Here,
\begin{align*}
     \varepsilon^{\Delta}_{m}
     = \hRms \sum_{j \neq j'} \hnjjpms \DIjjpms
\end{align*}
is the local approximation error arising from the non-commutativity noise,
\begin{align*}
    \epprmn{\nm N}
    = \hRms\big(\sqrt{I+\hZmos^2} - I - \hZmos - \tfrac12 \hZmos^2\big)
\end{align*}
is the local truncation remainder of the second-order Taylor expansion,
and
\begin{align*}
    \epprm
    = \hRms \hZmos^2
    - \hRms \sum_{j,j'=1}^d \SS_{j,j'}^m \tIjjpms
    - \hRms \sum_{j=1}^d \SS_{j,j'}^m (\Dms t)^2
\end{align*}
collects higher-order terms. Therefore, we have
\begin{align}\label{eq:dif-non-com}
       & \E\big[ \sup_{0 \le t\le T}\| \hatR(t)- \bfR_\delta(t) \|^2 \big]  \nonumber \\
   \le & 3\EsupN{\varepsilon^{\Delta}_{m}}    + 3\EsupN{\epprm}  + 3\EsupN{ \epprmn{\nm N}}.
\end{align}
Since $\EsupN{ \epprmn{\nm N}}$ has been shown to contribute at most $\mathcal O(\delta^{2})$ accumulated error in Appendix~\ref{app:lemma1}. It is sufficient to show
\begin{align*}
    \EsupN{\varepsilon^{\Delta}_{m}}  = \EsupN{\epprm} = \calO(\delta^2).
\end{align*}

\noindent\textbf{Step 2: Estimating $\varepsilon^{\Delta}_{m}$.}

By~\cref{lem:DI1-DI2} we know the partial sum  $ \sum_{m=0}^{\sumub} \varepsilon^{\Delta}_{m}$ is a martingale. Moreover, since $\hRms \in \SOn$ and each $\hnjjpms$ is uniformly bounded by a constant $C$ under~\cref{asm:drift,asm:diffusion}, we obtain
\begin{align*}
     \EsupN{\varepsilon^{\Delta}_{m}} = \E[\|\hRms \sum_{j \neq j'} \hnjjpms \DIjjpms\|^2 ] \le d^2 C^2\sum_{j \neq j'}\E[\DIjjpms^2]
\end{align*}
Since $h \ge \delta^{-\thalf}$, \cref{lem:DI1-DI2} implies  $\E[(\DIjjpms)^2] = \calO(\delta^3)$. Applying Lemma~\ref{lemma:error-sum-mar}, we have
\begin{align}\label{eq:err-DI}
      \EsupN{\varepsilon^{\Delta}_{m}} = \EsupN{\hRms \sum_{j \neq j'} \hnjjpms \DIjjpms } = \calO(\delta^2),
\end{align}

\noindent\textbf{Step 3: Estimating $\epprm$.}

We write $\epprm$ as follows.
\begin{align}\label{eq:varep-expansion}
     &\qquad \qquad \qquad \qquad \qquad \qquad\epprm= \sum_{i=1}^9\epprmn{i}, \nonumber\\
     &\begin{cases}
        \epprmn{1} = \hatR_m\sum_{j=1}^d \SSO(\hBIms, \hBjms) (\Dms t \cdot \Dms W_j) \\
        \epprmn{2} = \hatR_m\sum_{j=1}^d\sum_{l\neq l'}  \SSO(\hBjms ,\hDllpms) (\Dms W_j \cdot \tIllpms) \\
        \epprmn{3} = \hatR_m\sum_{j=1}^d\sum_{l=1}^d \SSO (\hBjms, \hDllms)(\Dms W_j) ( (\Dms W_l)^2 - \Dms t)),\\
        \epprmn{4} = \hatR_m \SSO(\hBIms,\hBIms) (\Dms t)^2  \\
        \epprmn{5} = \Big( \hatR_m \sum_{j\neq j'}  \hDjjpms  (\tIjjpms)\Big)^2  \\
        \epprmn{6} = \Big( \hatR_m \sum_{j=1}^d\hDjjms \big( (\Dms W_j)^2 - \Dms t \big) \Big)^2 \\
        \epprmn{7} = \hatR_m\sum_{j\neq j'}  \SSO(\hBIms, \hDjjpms) (\Dms t \cdot \tIjjpms) \\
        \epprmn{8} = \hRms  \sum_{j} \SSO(\hBIms, \hDjjms) \Big(\Dms t \cdot \big((\Dms W_j)^2 - \Dms t\big)\Big) \\
        \epprmn{9} = \hatR_m \sum_{j\neq j'} \sum_{l} \SSO( \hDjjpms, \hDllms) \tIjjpms((\Dms W_l)^2 - \Dms t)).
    \end{cases}
\end{align}
In~\cref{eq:varep-expansion}, all terms were bounded in~\cref{app:lemma1} \emph{except} the terms related to $\tIjjpms$, i.e.,
\begin{align*}
    \epprmn{2}, \epprmn{5},\epprmn{7},\epprmn{9},
\end{align*}
now we analyze the above terms.

We first estimate $\epprmn{2}$. The partial sum
    \begin{align*}
        \tilde H^{(2)}_t = \sum_{m=1}^{\sumub} \epprmn{2}
    \end{align*}
    is an $\mathcal{F}_t$-martingale. In addition, since $\hRms \in \SOn$ and each  $\SSO(\hBjms ,\hDllpms)$ is uniformly bounded  under Assumptions~\ref{asm:drift} and \ref{asm:diffusion}, we obtain
    \begin{align*}
       \E[ \|(\epprmn{2})^2\|] \le & C\sum_{j=1}^d\sum_{l=1}^d \E[\| \Dms W_j \tIllpms\|^2] \\
        \le& C\sum_{j=1}^d\sum_{l=1}^d \E[\| \Dms W_j \Illpms\|^2] +  C\sum_{j=1}^d\sum_{l=1}^d \E[\|  \Dms W_j \DIllpms\|^2] = \calO(\delta^3).
    \end{align*}
     where the last equality follows from~\cref{lem:I2W2-I2W4,lem:DI1W-DI2W2-DI2W4}. Hence, applying~\cref{lemma:error-sum-mar}, we have
    \begin{align*}
        \EsupN{\epprmn{2}}  = \calO(\delta^2)
    \end{align*}
    
Second, we estimate $\varepsilon_m^{(5)}$.  
Under the same boundedness assumptions, we have
\begin{align*}
\E[(\epprmn{5})^2] \le & C\sum_{j\neq j}\E[(\tIjjpms)^4] \le C\sum_{j\neq j}\E[(\Ijjpms)^4] +  C\sum_{j\neq j}\E[(\DIjjpms)^4]  = \mathcal{O}(\delta^4),
\end{align*}
where the last equality follows from~\cref{lem:I2N,lem:DI4}, which immediately implies
\begin{align*}
    \E\Bigl[\sup_{0 \le t\le T}
      \Bigl\|\sum_{m=0}^{\sumub}\varepsilon_m^{(5)}\Bigr\|^2\Bigr]
    = \mathcal{O}(\delta^2).
\end{align*}
Furthermore, by~\cref{lem:I2N,lem:DI1-DI2}, we know that 
\[\E[(\tIjjpms)^2] \le 2 \E[(\Ijjpms)^2] + 2 \E[(\DIjjpms)^2] = \mathcal{O}(\delta^2).\]  
Consequently,
\begin{align*}
    \E[(\tIjjpms)^2(\Dms t)^2] = \mathcal{O}(\delta^4),
\end{align*}
which leads to
\begin{align*}
    \E\Bigl[\sup_{0 \le t\le T}
      \Bigl\|\sum_{m=0}^{\sumub}\varepsilon_m^{(7)}\Bigr\|^2\Bigr]
    = \mathcal{O}(\delta^2).
\end{align*}

Finally, using the same argument, we obtain
\begin{align*}
  \E[(\tIjjpms)^2(\Dms W_j)^4]  =\E[(\tIjjpms)^2(\Dms t)^2]
    = \mathcal{O}(\delta^4),
\end{align*}
which yields
\begin{align*}
    \E\Bigl[\sup_{0 \le t\le T}
      \Bigl\|\sum_{m=0}^{\sumub}\varepsilon_m^{(9)}\Bigr\|^2\Bigr]
    = \mathcal{O}(\delta^2),
\end{align*}
Finally, we have
\begin{align}\label{eq:epprm}
    \EsupN{\epprm} \le 9\sum_{i=1}^9\EsupN{\epprmn{i}}.
\end{align}

\noindent\textbf{Step 4: Putting Together.}

Putting~\cref{eq:epprm,eq:err-DI} into~\cref{eq:dif-non-com}, we have
\begin{align}\label{eq:dif-non-com-2}
     \E\big[ \sup_{0 \le t\le T}\| \hatR(t)- \bfR_\delta(t) \|^2 \big] = \calO(\delta^2).
\end{align}

Given the error in~\cref{eq:dif-non-com-2} and following the same proof as in~\cref{lemma:2}, we have
\begin{align*}
    \E\big[ \sup_{0 \le t\le T} \|  \bfR_\delta(t) - \bfR(t)\|^2_F \big]  = \calO(\delta^2),
\end{align*}
and, finally, we conclude that
\begin{align*}
     \E\big[ \sup_{0 \le t\le T}\| \hatR(t)- \bfR(t) \|^2 \big]^{\half} = \calO(\delta),
\end{align*}
which completes our proof.

\end{proof}

\section{Proof of Theorem~\ref{thm:tasp-se}}\label{app:tasp-se} 
\begin{proof}

To find the normal adjustment, we seek a normal element
$
\nm U=(\nm C,\nm b)\in\senp
$
such that 
\[
I+\tg V+\nm U
=
\begin{bmatrix}
I+\tg Z+\nm C & v\\
\nm b^{\top} & 1
\end{bmatrix}\in \SEn.
\]
The group constraint imposes conditions that
\[
I+\tg Z+\nm C\in\SOn,
\qquad
\nm b^{\top}=0 .
\]
Therefore, from~\cref{lemma:correction}, we have
\begin{align}
    \nm C
=
\sqrt{I-\tg Z^{\top}\tg Z}-I,
\end{align}
Hence,
$
\nm U=(\sqrt{I-\tg Z^{\top}\tg Z}-I,0)
$
is the normal adjustment for TaSP over $\SEn$.

Now we parametrize the solution of~\eqref{eq:I-SDE-sen} as
\[
\bfE(t)
=
\bfE_0\bigl(I+\tg V(t)+\nm U(\tg V(t))\bigr).
\]
Differentiating on both sides gives
\[
d\bfE(t) = \bfE_0\bigl(d\tg V(t)+d\nm U(t)\bigr).
\]
Because $\nm U(t)\in\senp$, projection onto the tangent space
$T_{\bfE_0}\SEn=\bfE_0\sen$ eliminates the normal component
\[
\PEtg (d\bfE(t)) = \bfE_0d\tg V(t).
\]

Applying the projection operator to the SDE~\eqref{eq:I-SDE-sen} yields
\[
\PEtg(d\bfE(t)\bigr) = \Bigl( \B_I(\bfE(t))dt + \sum_{j=1}^{d}\B_j(\bfE(t))dW_j(t) \Bigr).
\]
Left‐multiplying by $\bfE_0^{-1}$ gives the dynamics for the tangent motion $ \tg V(t) $, i.e.,
\[
d\tg V(t)
=
\B_I(\bfE(t))dt
+
\sum_{j=1}^{d}\B_j(\bfE(t))dW_j(t).
\]

Writing $\B_I(\bfE)=(\B_I^{\bfR}(\bfE),\B_I^{\bf p}(\bfE))$,$\B_j(\bfE)=(\B_j^{\bfR}(\bfE),\B_j^{\bf p}(\bfE))$ for $ j \in [d],$
we finally obtain the coupled SDEs
\begin{align*}
\begin{cases}
    d\tg Z(t) &= \B_I^{\bfR}(\bfE(t))dt+\sum_{j=1}^{d}\B_j^{\bfR}(\bfE(t))dW_j(t),\\
    dv(t) &= \B_I^{\bf p}(\bfE(t))dt+\sum_{j=1}^{d}\B_j^{\bf p}(\bfE(t))dW_j(t).
\end{cases}
\end{align*}
which completes the proof.

\end{proof}

\section{Proof of Theorem~\ref{thm:convergence-sen}}\label{app:conv-sen} 
\begin{proof}
We define the frozen-field Milstein series $\bfE_{\delta}(t)$ over $\SEn$ with $(\hnjjRms,\hnjjpPms) =\Der{\nabla}{\bf E \B_j}{\bf E \B_{j'}}{\hatE_m}$ as follows,
\begin{align*}
    \begin{cases}
          \bfR^E_\delta(t) = & \bfR_0 + \sum_{m=1}^{\sumub} \hRms \hBIRms \Dms t +  \half \sum_{m=1}^{\sumub}  \hRms\sum_{j=1}^d (\hBjRms)^2 \Dms t  \\ 
            & + \sum_{m=1}^{\sumub}\hRms\sum_{j=1}^d \hBjRms \Dms W_j(t)  + \half\sum_{m=1}^{\sumub}\hRms\sum_{j,j'=1}^d \hnjjpRms \Ijjpms, \\
          {\bf p}^E_{\delta}(t) = & \bf p_0 + \sum_{m=1}^{\sumub} \hRms \hBIPms \Dms t \\ 
            & + \sum_{m=1}^{\sumub}\hRms\sum_{j=1}^d \hBjPms \Dms W_j(t) + \half\sum_{m=1}^{\sumub}\hRms\sum_{j,j'=1}^d \hnjjpPms \Ijjpms. \\
    \end{cases}
\end{align*}

Therefore, the error splits as
\begin{align}\label{eq:global-split}
    \E\bigl[\sup_{0\le t\le T}\|\hatE(t)-\bfE(t)\|^2\bigr]
    \le 
    2\E\bigl[\sup_{0\le t\le T}\|\hatE(t)-\bfE_\delta(t)\|^2\bigr]
    +
    2\E\bigl[\sup_{0\le t\le T}\|\bfE_\delta(t)-\bfE(t)\|^2\bigr].
\end{align}

Write the first term in~\eqref{eq:global-split} component-wise, we have
\begin{align*}
    \E\bigl[\sup_{0\le t\le T}\|\hatE(t)-\bfE_\delta(t)\|^2\bigr]
    \le 
    \E\bigl[\sup_{0\le t\le T}\|\hatR^E(t)-\bfR^E_\delta(t)\|^2+\sup_{0\le t\le T}\|\hatp^E(t)-\bf p^E_\delta(t)\|^2\bigr].
\end{align*}
The rotational component of the TaSP–CM iterate $\hatR^E(t)$ and the frozen-field Milstein series $\bfR^E_\delta(t)$ on the $\SEn$ block are identical to those on $\SOn$ that analyzed in ~\cref{lemma:1}, hence
\[
\E\bigl[\sup_{0\le t\le T}\|\hatR^E(t)-\bfR^E_\delta(t)\|^2\bigr]=\calO(\delta^{2}).
\]
For the translational part, as the sequence $\hatp^E(t)$ runs in the flat space, we have that $\hnjjRms = \hDjjRms$. Therefore
$
\hatp(t)=\bf p_\delta(t)
$
for every $t$. Therefore, we have
\begin{align}\label{eq:num-vs-mil}
    \E\bigl[\sup_{0\le t\le T}\|\hatE(t)-\bfE_\delta(t)\|^2\bigr]=\E\bigl[\sup_{0\le t\le T}\|\hatR^E(t)-\bfR^E_\delta(t)\|^2\bigr] =\calO(\delta^{2}).
\end{align}

Given~\cref{eq:num-vs-mil}, we follow the same argument in~\cref{lemma:2} to get
\begin{align}\label{eq:mil-vs-exact}
    \E\bigl[\sup_{0\le t\le T}\|\bfE_\delta(t)-\bfE(t)\|^2\bigr]=\calO(\delta^{2}).
\end{align}

Substituting~\eqref{eq:num-vs-mil} and~\eqref{eq:mil-vs-exact}
into~\eqref{eq:global-split} yields
\[
\E\bigl[\sup_{0\le t\le T}\|\hatE(t)-\bfE(t)\|^2\bigr]^{\half}=\calO(\delta),
\]
which completes our proof. $\hfill\qed$
\end{proof}
\end{appendices}
\bibliographystyle{plain}
\bibliography{ref}

\end{document}